%% file: Novikov.tex
\documentclass{amsart}
\input{top}

\title[Novikov cohomology]{Novikov cohomology, finite domination, and cohomological dimension}
\author{Sam P.\ Fisher}
\email{samuel.p.fisher@gmail.com}

\begin{document}

\begin{abstract}
    We introduce the $\Sigma^*$-invariant of a group of finite type, which is defined to be the subset of non-zero characters $\chi \in \H^1(G;\R)$ with vanishing associated top-dimensional Novikov cohomology. We prove an analogue of Sikorav's Theorem for this invariant, namely that $\cd(\ker \chi) = \cd(G) - 1$ if and only if $\pm \chi \in \Sigma^*(G)$ for integral characters $\chi$. This implies that cohomological dimension drop is an open property among integral characters. We also study the cohomological dimension of arbitrary co-Abelian subgroups. The techniques yield a short new proof of Ranicki's criterion for finite domination of infinite cyclic covers, and in a different direction, we prove that the algebra of affiliated operators $\mathcal U(G)$ of a RFRS group $G$ has weak dimension at most one if and only if $G$ is an iterated (cyclic or finite) extension of a free group.
\end{abstract}

\maketitle

\section{Introduction}

In \cite{BNSinv87}, Bieri, Neumann, and Strebel introduced the $\Sigma$-invariant of a finitely generated group $G$. The invariant, denoted $\Sigma(G)$, is a subset of $\H^1(G;\R) \smallsetminus \{0\}$, and controls the finiteness properties of co-Abelian subgroups of $G$. We call elements of $\H^1(G;\R)$ \emph{characters}; these are just homomorphisms $\chi \colon G \rightarrow \R$. If $\chi(G) \subseteq \Q$, then we say that $\chi$ is \emph{integral} (because in this case $\chi(G) \cong \Z$). The set of all non-zero characters which vanish on a given subgroup $N$ is denoted by $S(G,N)$. The following celebrated result from \cite{BNSinv87} establishes the fundamental properties of the $\Sigma$-invariant.

\begin{thm}[Bieri--Neumann--Strebel, {\cite{BNSinv87}}]\label{introthm:BNS}
    Let $G$ be a finitely generated group. The following properties hold:
    \begin{enumerate}
        \item $\Sigma(G)$ is open in $\H^1(G;\R) \smallsetminus \{0\}$;
        \item\label{introitem:coabelian_BNS} if $N \trianglelefteqslant G$ is a normal subgroup such that $G/N$ is Abelian, then $N$ is finitely generated if and only if $S(G,N) \subseteq \Sigma(G)$.
    \end{enumerate}
\end{thm}

In particular, an integral character $\chi \colon G \rightarrow \Z$ has finitely generated kernel if and only if $\pm \chi \in \Sigma(G)$. Recall that a map to $\Z$ with finitely generated kernel is called an \emph{algebraic fibration} of $G$. \cref{introthm:BNS} has the following well known corollary.

\begin{cor}\label{introcor:fibring_is_open}
    Algebraic fibring is an open property. More precisely, given a finitely generated group $G$, the set
    \[
        \{\chi \colon G \longrightarrow \Q \ : \ \ker \chi \ \textnormal{is finitely generated} \} \smallsetminus \{0\}
    \]
    is an open subset of $\H^1(G;\Q) \smallsetminus \{0\}$.
\end{cor}

The $\Sigma$-invariant is well-understood in certain cases (e.g.\ for right-angled Artin groups \cite{MeierMeinertVanWyk1998,BuxGonzalez1999}), but for most classes of groups it is notoriously difficult to compute. A useful tool is provided by Sikorav's Theorem, which gives an elegant homological criterion for a character to lie in the $\Sigma$-invariant.

\begin{thm}[Sikorav, {\cite{SikoravThese}}]
    If $G$ is a finitely generated group, then $\chi \in \Sigma(G)$ if and only if $\H_1(G;\widehat{\Z[G]}^\chi) = 0$.
\end{thm}

The object $\widehat{\Z[G]}^\chi$ appearing in the theorem is the \emph{Novikov ring} associated to $\chi$; it is a certain completion of the group ring $\Z[G]$ along $\chi$, and can be thought of as the ring of infinite series of elements in $G$ which go to infinity in the direction of $\chi$ (see \cref{ex:Novikov} for a precise definition). Sikorav's Theorem plays a central role in Kielak's proof that a finitely generated RFRS group virtually algebraically fibres if and only if it has vanishing first $L^2$-Betti number \cite{KielakRFRS}. His strategy was to related the $L^2$-homology of a RFRS group $G$ with the homologies of its finite-index subgroups with coefficients in Novikov rings associated to various characters. 

Sikorav's Theorem has many generalisations and variations. For instance, the kernel of an integral character $\chi \colon G \rightarrow \Z$ being of type $\FP_n$ is equivalent to the vanishing of $\H_i(G;\widehat{\Z[G]}^{\pm\chi}) $ for all $i \leqslant n$ (see \cite{BieriRenzValutations}, as well as Schweitzer's appendix to \cite{BieriDeficiency}). This result also holds if one replaces the coefficient ring $\Z$ by an arbitrary ring $R$, and type $\FP_n$ by type $\FP_n(R)$ (see, e.g., \cite[Theorem 5.3]{Fisher_Improved}). A closely related result is Ranicki's criterion \cite[Theorem 1]{Ranicki_NovikovFiniteDomination}, which states that an infinite cyclic cover of a compact CW complex $X$ is finitely dominated if and only if $\H_i(\pi_1(X); \widehat{\Z[\pi_1(X)]}^{\pm\chi}) = 0$ for all $i$, where $\chi \colon \pi_1(X) \rightarrow \Z$ is the map associated to the cyclic cover (Ranicki's original result was for spaces $X$ such that $\pi_1(X) = G \times \Z$ and $\chi$ being the projection onto the $\Z$ factor, but it holds in the generality stated here). Yet another result in this direction is Latour's theorem, which, for a manifold $M$ of dimension at least $6$, characterises the existence of closed non-singular $1$-forms in a class $\chi \in \H^1(M;\R)$ in terms of the vanishing of Novikov homology associated to $\chi$ and of a $K$-theoretic invariant \cite{Latour_1formes}.

All the results mentioned above directly or indirectly relate vanishing Novikov homology of a character $\chi$ to the finiteness properties of its kernel. In \cite{Fisher_freebyZ}, the author showed that Novikov \emph{co}homology controls a different property of the kernel, namely its cohomological dimension. This was done with the goal of proving that RFRS groups of cohomological dimension two are virtually free-by-cyclic if and only if they have vanishing second $L^2$-Betti number. The main new input was a criterion stating that if a group $G$ of type $\FP$ has $\H^{\cd(G)}(G;\widehat{\Z[G]}^{\pm \chi}) = 0$, then $\cd(\ker \chi) = \cd(G) - 1$. One of the main purposes of this article is to prove a converse, thus showing that vanishing top-dimensional Novikov cohomology of integral characters completely determines the cohomological dimension of the kernel. In fact, the result holds in the generality of chain complexes of projective modules.

\begin{thmx}\label{introthm:cohom_drop_complex}
    Let $G$ be a group, let $\chi \colon G \rightarrow \Z$ be an integral character, and let $C_\bullet$ be a chain complex of projective $\Z[G]$-modules such that $C_i = 0$ for all $i > n$, for some $n \in \Z$, and assume that $C_n$ is finitely generated. The following are equivalent:
    \begin{enumerate}
        \item\label{introitem:top_dim_nov_zero} $\H^n(C_\bullet; \widehat{\Z[G]}^{\pm\chi}) = 0$;
        \item\label{introitem:dim_drop} $C_\bullet$ is chain homotopy equivalent over $\Z[\ker \chi]$ to a chain complex $D_\bullet$ of projective $\Z[\ker \chi]$-modules such that $D_i = 0$ for all $i > n-1$.
    \end{enumerate}
\end{thmx}

\cref{introthm:cohom_drop_complex} is proved by analysing the long exact sequence in cohomology induced by a special short exact sequence involving the Novikov ring (see \cref{obs:SES}). This short exact sequence was communicated to the author by Andrei Jaikin-Zapirain as a tool to simplify the original proof of \cite[Theorem A]{Fisher_freebyZ}. It was also used in \cite{FisherItalianoKielak_PDfibre} to study the algebraic fibres of Poincar\'e-duality groups. The most general form of \cref{introthm:cohom_drop_complex} will be given in \cref{cor:cd_chain_complex}; in particular, the theorem still holds if one replaces the coefficient ring $\Z$ with any ring $R$ (we always consider maps $G \to \Z$, of course). The crux of the implication \ref{introitem:dim_drop} $\Rightarrow$ \ref{introitem:top_dim_nov_zero} is an observation about general modules over Novikov rings: if $M$ is a module over $\widehat{\Z[G]}^\chi$ (and $\chi \neq 0$), then either $M = 0$ or $M$ is uncountable (\cref{lem:countable}).

\cref{introthm:cohom_drop_complex} has the following immediate consequence for groups, proving that top-dimensional Novikov cohomology is a perfect obstruction for cohomological dimension drop. Motivated by Sikorav's Theorem, we define the \emph{$\Sigma^*$-invariant} of a group $G$ of type $\FP$ by 
\[
    \Sigma^*(G) \ := \ \{\chi \in \H^1(G;\R) \smallsetminus \{0\} \ : \ \H^{\cd(G)}(G;\widehat{\Z[G]}^\chi) = 0 \}.
\]

\begin{corx}\label{introcor:if_and_only_if}
    Let $G$ be a group of type $\FP$. If $\chi \colon G \rightarrow \Z$ is an integral character, then $\cd(\ker \chi) = \cd(G) - 1$ if and only if $\pm\chi \in \Sigma^*(G)$.
\end{corx}

It is well-known that vanishing Novikov homology is an open property for groups of finite type; this applies just as well to Novikov cohomology, so we obtain the following corollary (see \cref{cor:open}), which is an analogue of \cref{introcor:fibring_is_open} for cohomological dimension drop.

\begin{corx}\label{introcor:open}
    If $G$ is a group of type $\FP$, then cohomological dimension drop is an open property among integral characters. More precisely, the set
    \[
        \{\chi \colon G \longrightarrow \Q \ : \ \cd(\ker \chi) = \cd(G) - 1\}
    \]
    is an open subset of $\H^1(G;\Q)$. In particular, if $\cd(G) = 2$, then the set of integral characters with free kernel is open in $\H^1(G;\Q)$.
\end{corx}

From a slightly stronger version of \cref{introthm:cohom_drop_complex}, we obtain a result on the $L^2$-invariants of RFRS groups which generalises \cite[Theorem A]{Fisher_freebyZ}. Let $\mathcal U(G)$ be the algebra of operators affiliated to the group von Neumann algebra of $G$. The $L^2$-Betti numbers of $G$ are defined by $b_i^{(2)}(G) = \dim_{\mathcal U(G)} \H_i(G;\mathcal U(G))$, where $\dim_{\mathcal U(G)}$ denotes the von Neumann dimension (see \cite{Luck02} for details). The \emph{weak dimension} of an $R$-module $M$ is the supremal integer $n$ for which there exists an $R$-module $N$ such that $\Tor_n^R(M,N) \neq 0$.

\begin{thmx}\label{introthm:RFRS}
    Let $G$ be a RFRS group of type $\FP(\Q)$. The following are equivalent:
    \begin{enumerate}
        \item $\mathcal U(G)$ is of weak dimension at most one as a $\Q[G]$-module;
        \item for every subgroup $H \leqslant G$, we have $b_i^{(2)}(H) = 0$ for all $i > 1$;
        \item $G$ is an iterated (cyclic or finite) extension of a free group.
    \end{enumerate}
\end{thmx}

\cref{introthm:RFRS} will be used to prove that RFRS groups admitting an Abelian hierarchy are iterated (cyclic or finite) extensions of free groups, extending a result of Hagen and Wise \cite{HagenWise_freebyZ} (see \cref{cor:Abelian_hierarchy}). The weak dimension one property is a crucial step in Jaikin-Zapirain and Linton's proof that group algebras of one-relator groups are coherent \cite{JaikinLinton_coherence}. In a similar vein, a conjecture of Wise predicts that groups of cohomological dimension at most two are coherent if and only if they have vanishing second $L^2$-Betti number. If $G$ is RFRS and $\cd(G) \leqslant 2$, then an easy consequence of \cite[Theorem A]{Fisher_freebyZ} is that the vanishing of the second $L^2$-Betti number is equivalent to $\mathcal U(G)$ being of weak dimension at most one. Motivated by all this, we make the following prediction for coherence in the class of finite type RFRS groups.

\begin{conj}
    If $G$ is a RFRS group of type $\FP$, then the following are equivalent:
    \begin{enumerate}
        \item $G$ is coherent;
        \item $\Q[G]$ is coherent;
        \item $\mathcal U(G)$ is of weak dimension one as a $\Q[G]$-module.
    \end{enumerate}
\end{conj}

The method of proving \cref{introthm:cohom_drop_complex} also yields a new proof of Ranicki's Criterion \cite{Ranicki_NovikovFiniteDomination} for finite domination of cyclic covers (see \cref{thm:ranicki}), as well as a new proof of the implication
\[
    \H_1(G;\widehat{\Z[G]}^{\pm\chi}) = 0 \qquad \Longrightarrow \qquad \ker \chi \ \text{is finitely generated}
\]
(see \cref{prop:Sikorav}). This is the more subtle direction of Sikorav's Theorem, and is usually the one used in practice (such as in Kielak's fibring theorem). The new proofs are more conceptual, as there is no longer the need to deal with individual Novikov cycles. One instead works with long exact sequences of homology modules and other standard results from homological algebra, and uses the vanishing of Novikov homology to show that the homology of the kernel commutes with direct products. By Brown's homological characterisation of finite domination \cite{Brown_HomCritFinite}, this yields finite generation of the kernel. Thus, the method links two well-known criteria for finite domination of chain complexes: Brown's and Ranicki's. Another satisfying aspect of this proof is that it unifies the points of view on Novikov cohomology controlling cohomological dimension and Novikov homology controlling finiteness properties of the cyclic covers; these phenomena follow from dual arguments. A drawback, however, is that it is crucial to work with integral characters and the vanishing of Novikov homology associated to both $\chi$ and $-\chi$ simultaneously. We do not obtain an new proof of the fact that $\chi \in \Sigma(G)$ when $\H_1(G;\widehat{\Z[G]}^\chi) = 0$ for general characters $\chi \colon G \rightarrow \R$.

From the Lyndon--Hochschild--Serre spectral sequence, it is easy to see that the cohomological dimension of $\ker(\chi \colon G \rightarrow \Z)$ is at least $\cd(G)-1$, so vanishing of Novikov cohomology in low codimensions cannot imply further cohomological dimension drop of the kernel (we will actually see that vanishing of low-codimensional Novikov cohomology implies the functors $\H^i(\ker \chi; - )$ commute with direct limits in the same low codimensions, which can be seen as a type of finiteness property). However, it still makes sense to ask whether Novikov cohomology can control further cohomological drop of the kernel of a non-integral character. We define the higher $\Sigma^*$-invariants of a group $G$ of finite type by
\[
    \Sigma_m^*(G) \ := \ \{ \chi \in \H^1(G;\R) \smallsetminus \{0\} \ : \ \H^i(G;\widehat{\Z[G]}^\chi) = 0 \ \text{for all} \ i > \cd(G)-m\}
\]
for $m \geqslant 0$. Hence, $\Sigma_0^*(G) = \H^1(G;\R) \smallsetminus \{0\}$ and $\Sigma_1^*(G) = \Sigma^*(G)$. We obtain the following sufficient condition for a map to a free Abelian group to have kernel of the minimal theoretically allowed cohomological dimension.

\begin{thmx}\label{introthm:coabelian}
    Let $G$ be a group of type $\FP$. Let $N \trianglelefteqslant G$ be a normal subgroup such that $G/N \cong \Z^d$. If $S(G,N) \subseteq \Sigma_d^*(G)$, then $\cd(N) = n-d$.
\end{thmx}

Given the similarity with \cref{introthm:BNS}\ref{introitem:coabelian_BNS}, this raises the obvious question of whether the converse statement holds (\cref{introcor:if_and_only_if} is the important special case $d = 1$); it does not, see \cref{ex:bad_3_mfld}. \cref{introthm:coabelian} has the following corollary for Poincar\'e-duality groups.

\begin{corx}\label{introcor:PD}
    Let $G$ be a Poincar\'e-duality group of dimension $n$ and let $\chi \colon G \rightarrow \Z^d$. If $\ker \chi$ is of type $\FP_{d-1}$, then $\cd(G) = n-d$.
\end{corx}

The condition $\FP_0$ is vacuous, so in the case where $d = 1$, the corollary says that all kernels of maps onto $\Z$ drop in cohomological dimension. This is a very special case of Strebel's theorem, which states that all infinite-index subgroups of $\mathrm{PD}_n$-groups have cohomological dimension at most $n-1$. There are various situations where cohomological dimension is additive under group extensions. One of the most general results is Fel'dman's theorem \cite[Theorem 2.4]{Feldman71}, which states that if $1 \to N \to G \to Q \to 1$ is an extension of groups and $k$ is a field such that $N$ is of type $\FP(k)$ and $\cd_k(Q) < \infty$, then $\cd_k(N) + \cd_k(Q) = \cd_k(G)$. An interesting feature of \cref{introcor:PD} is that the finiteness property required on the kernel is much weaker than in Fel'dman's theorem. We will also see in \cref{cor:PD} that the result holds over arbitrary rings, whereas it is important in Fel'dman's theorem that $k$ be a field.

\subsection{Structure of the paper}

\cref{sec:prelims} introduces the necessary general homological algebra preliminaries for the article. In \cref{sec:fin_dom}, we will prove \cref{introthm:BNS} in its full generality. This means working with a general chain complex of finitely generated $R$-modules, and arbitrary Laurent series instead of Novikov rings. We will also give a new proof of Ranicki's Criterion and of the fact that vanishing Novikov homology detects finiteness properties. In \cref{sec:Sigma_inv}, we define the $\Sigma^*$-invariants of a group over a general ring, and prove \cref{introcor:if_and_only_if,introcor:fibring_is_open}, as well as \cref{introthm:coabelian} and \cref{introcor:PD}. In \cref{sec:RFRS}, we use Kielak's connection between Novikov and $L^2$-cohomology to prove \cref{introthm:RFRS}. Finally, in \cref{sec:one_rel} we look at the example of one-relator groups to show that the top-dimensional Novikov homology does not seem to be a useful invariant, in contrast to top-dimensional Novikov cohomology. We will see that torsion-free one-relator groups have vanishing second Novikov homology with respect to all characters, though this is far from the case for the second Novikov cohomology.

\subsection{Acknowledgments}

The author is grateful to Andrei Jaikin-Zapirain, Dawid Kielak, Ian Leary, and Boris Okun for useful comments. The author also thanks ICMAT, Andrei, Javier, Marco, and Sara, for their hospitality.

\section{Homological algebra preliminaries}\label{sec:prelims}

\begin{conv}
    All rings associative and unital, with $1 \neq 0$. Unless specified otherwise, $R$ always denotes an arbitrary ring and $G$ always denotes an arbitrary group.
\end{conv}

Let $R$ be a ring. A chain complex $\cdots \to C_{i+1} \to C_i \to \cdots$ of $R$-modules will usually be denoted by $C_\bullet$ and its boundary maps by $\partial_i \colon C_i \rightarrow C_{i-1}$ (or just by $\partial$ when the dimension is understood). Suppose that $C_\bullet$ is a chain complex of left $R$-modules. If $M$ is a right $R$-module, then $i$th \emph{homology} of $C_\bullet$ with coefficients in $M$ is $i$th homology of the chain complex $M \otimes_R C_\bullet$ and is denoted by $\H_i(C_\bullet; M)$. Similarly, if $M$ is a left $R$-module, then the $i$th \emph{cohomology} of $C_\bullet$ with coefficients in $M$ is the $i$th cohomology of the cochain complex $\Hom_R(C_\bullet,M)$ and is denoted by $\H^i(C_\bullet;M)$. There are analogous definitions when $C_\bullet$ consists of right $R$-modules.

If $C_\bullet$ and $D_\bullet$ are chain complexes, then a \emph{chain map} $f_\bullet \colon C_\bullet \rightarrow D_\bullet$ is a sequence of morphisms $f_i \colon C_i \to D_i$ such that $f_{i-1} \partial_i = \partial_i f_i$ for all $i$. Two chain maps $f_\bullet, g_\bullet \colon C_\bullet \rightarrow D_\bullet$ are said to be \emph{chain homotopic} if there is a sequence of morphisms $h_i \colon C_i \rightarrow D_{i+1}$ such that $f_i - g_i = h_{i-1} \partial_i + \partial_{i+1} h_i$ for all $i$. A chain map $f_\bullet \colon C_\bullet \rightarrow D_\bullet$ is a \emph{chain homotopy equivalence} if there is chain map $g_\bullet \colon D_\bullet \rightarrow C_\bullet$ such that $f_\bullet g_\bullet$ and $g_\bullet f_\bullet$ are each chain homotopic to the identity maps. In this case, we say that $C_\bullet$ and $D_\bullet$ are chain homotopic or of the same chain homotopy type.

Suppose that $C_\bullet$ is a chain complex of projective $R$-modules. We say that $C_\bullet$ is \emph{finitely dominated} or of \emph{finite type} if $C_\bullet$ is chain homotopic to a chain complex of finitely generated projective $R$-modules. If $C_i = 0$ for all $i < 0$ and $n$ is an integer, we say that $C_\bullet$ is of \emph{finite $n$-type} if $C_\bullet$ is chain homotopic to a chain complex $D_\bullet$ of projective modules such that $D_i$ is finitely generated for $i \leqslant n$ and $D_i = 0$ for all $i < 0$. A group $G$ is said to be of type $\FP_n(R)$ (resp.\ $\FP_\infty(R)$, resp.\ $\FP(R)$) if the trivial $R[G]$-module $R$ admits a projective resolution which is finitely generated in degrees at most $n$ (resp.\ finitely generated in all degrees, resp.\ of finite length and finitely generated in all degrees).

Brown gives the following characterisation of finite domination in terms of the (co)homology functors $\H_i(C_\bullet;-)$ and $\H^i(C_\bullet;-)$.

\begin{thm}[{\cite[Theorem 2]{Brown_HomCritFinite}}]\label{thm:Brown_crit_finiteness}
    Let $C_\bullet$ be a chain complex of projective $R$-modules such that $C_i = 0$ for all $i < 0$. The following are equivalent:
    \begin{enumerate}
        \item $C_\bullet$ is of finite $n$-type;
        \item for all directed systems of $R$-modules $\{M_j\}_{j\in J}$, the natural map
        \[
            \varinjlim \H^i(C_\bullet; M_j) \longrightarrow \H^i(C_\bullet; \varinjlim M_j)
        \]
        is an isomorphism for $i < n$ and a monomorphism for $i = n$;
        \item for all directed systems $\{M_j\}_{j\in J}$ of $R$-modules with $\varinjlim M_j = 0$, we have $\varinjlim \H^i(C_\bullet; M_j) = 0$ for $i \leqslant n$;
        \item for any family $\{M_j\}_{j \in J}$ of $R$-modules, the natural map 
        \[
            \H_i(C_\bullet; \prod M_j) \longrightarrow \prod \H_i(C_\bullet; M_j)
        \]
        is an isomorphism for $i < n$ and an epimorphism for $i = n$;
        \item for any index set $J$, the natural map $\H_i(C_\bullet; \prod_J R) \rightarrow \prod_J \H_i(C_\bullet; R)$ is an isomorphism for $i < n$ and an epimorphism for $i = n$.
    \end{enumerate}
\end{thm}

In particular, $C_\bullet$ is finitely dominated if and only if $\H^i(C_\bullet; - )$ preserves direct limits if and only if $\H_i(C_\bullet; - )$ preserves direct products.

\begin{defn}[Cohomological dimension]
    The \emph{cohomological dimension} $\cd(C_\bullet)$ of a chain complex $C_\bullet$ of $R$-modules is the supremal integer $n$ such that there exists an $R$-module $M$ with $\H^n(C_\bullet;M) \neq 0$.
\end{defn}

We will always use this definition of cohomological dimension of a chain complex, but it is nice to know that it also has the following characterisation when $C_\bullet$ consists of projective modules.

\begin{lem}
    If $C_\bullet$ is a chain complex of projective $R$-modules, then $C_\bullet$ is chain homotopy equivalent to a complex $D_\bullet$ such that $D_i = 0$ for all $i > n$ if and only if $\cd(C_\bullet) \leqslant n$.
\end{lem}
\begin{proof}
    If $C_\bullet$ is chain homotopy equivalent to $D_\bullet$ as in the statement, then it is clear that $\H^i(C_\bullet;M) = 0$ for all $R$-modules $M$ and all $i > n$, since chain homotopy equivalences induce isomorphisms on homology. Hence $\cd(C_\bullet) \leqslant n$.

    Conversely, suppose that $\cd(C_\bullet) \leqslant n$. Let $M = \im(\partial_{n+1}) \leqslant C_n$. By assumption, the cocycle $z \colon C_{n+1} \xrightarrow{\partial} M$ is zero in $\H^{n+1}(C_\bullet;M)$, so there is a cochain $c \colon C_n \to M$ such that $c \partial = z$. This implies that $M$ is a direct summand of $C_n$; write $C_n = M \oplus P$. Then $C_\bullet$ splits as the direct sum of the chain complexes
    \[
        \{ \cdots \to C_{n+2} \to C_{n+1} \to M \to 0 \to \cdots \} \oplus \{ \cdots \to 0 \to P \to C_{n-1} \to C_{n-2} \to \cdots \} 
    \]
    of projectives modules, where $M$ and $P$ both lie in degree $n$. We claim that the complex on the left (call it $E_\bullet$) is null-homotopic, which will prove the lemma.

    It is an easy exercise to check that $\H^i(E_\bullet; L) = 0$ for all integers $i$ and all $R$-modules $L$. Let $N = \im(\partial_{n+2})$ and consider the short exact sequence
    \[
        0 \longrightarrow N \longrightarrow C_{n+1} \longrightarrow M \longrightarrow 0.
    \]
    By assumption, the cocycle $w \colon C_{n+2} \xrightarrow{\partial} N$ is zero in cohomology, so there is some $d \colon C_{n+1} \rightarrow N$ such that $w = d\partial$. Hence, the short exact sequence above is split, so
    \[
        E_\bullet = \{\dots \to C_{n+2} \to N \to 0 \to \cdots\} \oplus \{\cdots \to 0 \to M \to M \to 0 \to \cdots \}.
    \]
    Again, the complex $\{\dots \to C_{n+2} \to N \to 0 \to \cdots\}$ has vanishing cohomology with all coefficients, so we can split off $\{\cdots \to 0 \to N \to N \to 0 \to \cdots\}$ as a direct summand, whose complement has vanishing cohomology with all coefficients. Continuing in this way show shows that $E_\bullet$ is a direct sum of complexes of the form $\{0 \to L \to L \to 0\}$ and is therefore null-homotopic. \qedhere
\end{proof}

The Eilenberg--Watts theorem will be useful to us in induction arguments, including in the proof of \cref{lem:acyclic_ring_module} just below. Morally, it says that tensoring is the only right-exact additive functor that preserves direct sums.

\begin{thm}[{\cite{Watts_EWtheorem}}]\label{thm:EW_thm}
    Let $R$ and $S$ be rings, and let $F \colon \mathsf{Mod}_R \rightarrow \mathsf{Mod}_S$ be an additive functor. If $F$ is right-exact and preserves arbitrary direct sums, then there is a natural isomorphism
    \[
        F(-) \ \cong \ - \otimes_R F(R),
    \]
    where $R$ is viewed as a right module over itself via multiplication, and $F(R)$ is viewed as a left $R$-module via $r \cdot x := F(m_r)(x)$ (and $m_r$ is the left-multiplication by $r$ map).
\end{thm}

\begin{lem}\label{lem:acyclic_ring_module}
    Let $R$ be a ring, let $C_\bullet$ be a chain complex of projective left $R$-modules and let $m$ and $n$ be integers.
    \begin{enumerate}
        \item\label{item:homol_range} Suppose $C_i = 0$ for all $i < 0$. If $\H_i(C_\bullet; R) = 0$ for all $i \leqslant n$, then $\H_i(C_\bullet;M) = 0$ for all $i \leqslant n$ and all $R$-modules $M$.
        \item\label{item:cohomol_range} Suppose $C_i = 0$ for all $i > n$, that $\H^i(C_\bullet;-)$ commutes with direct limits for all $i \geqslant m$, and that $\H^i(C_\bullet; R) = 0$ for all $i \geqslant m$. Then $\H^i(C_\bullet;M) = 0$ for all $i \geqslant m$ and all $R$-modules $M$.
    \end{enumerate}
\end{lem}
\begin{proof}
    For \ref{item:homol_range}, see \cite[Lemma 2]{Brown_HomCritFinite}. For \ref{item:cohomol_range}, we will assume that $n \geqslant m$, since otherwise the statement is vacuous. We will prove that $\H^i(C_\bullet; M) = 0$ for all $R$-modules $M$ by reverse induction on $i$, starting with $i = n+1$ and terminating with $i = m$. There is nothing to prove in the base case, since $C_{n+1} = 0$ and therefore $\H^{n+1}(C_\bullet;M)$ trivially vanishes for all modules $M$.

    Suppose $\H^i(C_\bullet;M) = 0$ for all $R$-modules $M$ and for some $i > m$. Let 
    \[
        0 \longrightarrow M_0 \longrightarrow M_1 \longrightarrow M_2 \longrightarrow 0
    \]
    be a short exact sequence of modules. The associated long exact sequence in cohomology contains the portion 
    \[
        \H^{i-1}(C_\bullet; M_1) \longrightarrow \H^{i-1}(C_\bullet; M_2) \longrightarrow \H^i(C_\bullet; M_0) = 0,
    \]
    so $H^{i-1}(C_\bullet;-)$ is right exact. The functor $\H^{i-1}(C_\bullet;-)$ also preserves direct limits by assumption, so the Eilenberg--Watts theorem implies that $\H^{i-1}(C_\bullet;-)$ is naturally isomorphic to the functor $\H^{i-1}(C_\bullet;R) \otimes_R -$, which is zero by assumption. \qedhere
\end{proof}

\section{Twisted Laurent series and finite domination}\label{sec:fin_dom}

Let $R$ be a ring and let $\alpha \colon R \rightarrow R$ be a ring automorphism. The $\alpha$-twisted Laurent polynomial ring in one variable is denoted by $R_\alpha[t\inv,t]$. This is the Laurent polynomial ring over $R$, where we impose the rule $r t^n = t^n \alpha^n(r)$ for all $r \in R$ and all $n \in \Z$. 

Let $R_\alpha [t\inv,t \rrbracket$ denote the ring of Laurent series in the variable $t$, i.e.\ the ring of expressions of the form $\sum_{n = k}^\infty r_n t^n$ with $r_n \in R$ for all $n \geqslant k$, where $k$ is some integer. Multiplication is defined on $R_\alpha [t\inv,t \rrbracket$ by formally extending the multiplication on $R_\alpha[t\inv,t]$. Similarly, let $R_\alpha \llbracket t\inv,t]$ be the ring of formal series of the form $\sum_{n = -\infty}^k r_n t^n$ for some $k \in \Z$, i.e.\ the ring of Laurent series in the variable $t\inv$.

\begin{ex}[The Novikov ring]\label{ex:Novikov}
    The main example of the construction above will come from groups and maps to $\Z$. Let $\chi \colon G \rightarrow \R$ be an homomorphism with kernel $N$ and let $R$ be a ring. The \emph{Novikov ring} associated to this data is 
    \[
        \nov{R[G]}{\chi} \ = \ \left\{ \sum r_g g \ : \ |\{g : r_g \neq 0\} \cap \chi\inv(\left]-\infty, t \right])| < \infty \ \text{for all} \ t \in \R  \right\}.
    \]
    Now suppose the image of $\chi$ is infinite cyclic and let $t \in G$ map to a generator of $\chi(G)$. Denoting by $\alpha$ the automorphism of $R[N]$ induced by the conjugation action of $t$, we have $\nov{R[G]}{\chi} \cong R[N]_\alpha[t\inv, t\rrbracket$.
\end{ex}

In view of \cref{ex:Novikov}, we will often call the (co)homology of a complex $C_\bullet$ of $R_\alpha[t\inv,t]$-modules with coefficients in $R_\alpha[t\inv,t\rrbracket$ the Novikov cohomology of $C_\bullet.$

Given a right $R$-module $M$, let
\[
    M_\alpha\llbracket t\inv,t \rrbracket \ = \ \prod_{i\in\Z} Mt^i.
\]
The elements of $M_\alpha\llbracket t\inv,t \rrbracket$ are thought of as bi-infinite series $\sum_{i = -\infty}^\infty m_i t^i$, where $m_i \in M$ for each $i \in \Z$. We endow $M_\alpha\llbracket t\inv,t \rrbracket$ with the structure of a right $R_\alpha[t\inv,t]$-module by defining
\[
    \left(\sum_{i=-\infty}^\infty m_i t^i \right) \cdot rt^n \ = \ \sum_{i=-\infty}^\infty (m_i \alpha^{-i}(r)) t^{i+n}.
\]
We also define $M_\alpha[t\inv,t \rrbracket$ to be the submodule of $M_\alpha\llbracket t\inv,t \rrbracket$ consisting of series of the form $\sum_{i=k}^{\infty} m_i t^i$ for some $k \in \Z$.

Let $M_\alpha[t\inv,t\rrbracket$ be the $R_\alpha[t\inv,t]$-submodule of series of the form $\sum_{n = k}^{\infty} m_n t^n$ for some $k \in \Z$. Similarly, let $M_\alpha\llbracket t\inv,t]$ be the submodule of series of the form $\sum_{n = -\infty}^{k} m_n t^n$ for some $k \in \Z$. Crucially, $M_\alpha[t\inv,t\rrbracket$ is a right $R_\alpha[t\inv,t\rrbracket$-module, and $M_\alpha\llbracket t\inv,t]$ is a right $R_\alpha\llbracket t\inv,t]$-module. Finally, let $M_\alpha[t\inv,t]$ be the right $R_\alpha[t\inv,t]$-module of finitely supported sums $\sum_i m_i t^i$.

\begin{obs}\label{obs:SES}
    There is a short exact sequence
    \[
        \begin{tikzcd}[row sep = 0]
            0 \arrow[r] & M_\alpha[t\inv,t] \arrow[r] & M_\alpha\llbracket t\inv,t] \oplus M_\alpha[t\inv,t\rrbracket \arrow[r] & M_\alpha\llbracket t\inv,t \rrbracket \arrow[r] & 0 \\
            & x \arrow[r, maps to] & {(x,x)} & & \\
            & & {(x,y)} \arrow[r, maps to] & x-y &
        \end{tikzcd}
    \]
    of right $R_\alpha[t\inv,t]$-modules.
\end{obs}

Let $\iota \colon R \rightarrow R_\alpha[t\inv,t]$ be the natural inclusion. If $M$ is an $R_\alpha[t\inv,t]$-module, denote the restricted $R$-module by $\iota^! M$.

\begin{lem}[Shapiro's Lemma]
    Let $C_\bullet$ be a chain complex of $R_\alpha[t\inv,t]$-modules. For any right $R$-module $M$, we have
    \begin{enumerate}
        \item\label{item:shapiro} $\H_i(C_\bullet;M_\alpha[t\inv,t]) \cong \H_i(\iota^! C_\bullet; M)$ when $C_\bullet$ is a complex of left modules, and
        \item\label{item:co_shapiro} $\H^i(C_\bullet;M_\alpha\llbracket t\inv,t\rrbracket) \cong \H^i(\iota^! C_\bullet; M)$ when $C_\bullet$ is a complex of right modules
    \end{enumerate}
    for all $i \in \Z$.
\end{lem}
\begin{proof}
    The isomorphism in \ref{item:shapiro} is an immediate consequence of the isomorphism $M_\alpha[t\inv, t] \cong M \otimes_R R_\alpha[t\inv,t]$. For \ref{item:co_shapiro}, there is an isomorphism
    \[
        \Phi \colon \Hom_{R_\alpha[t\inv,t]}(L, M_\alpha\llbracket t\inv,t \rrbracket) \rightarrow \Hom_R(\iota^! L,M)
    \]
    for any right $R_\alpha[t\inv,t]$-module $L$; it is given by $\Phi(\varphi)(l) = \varphi(l)_0$, where $l \in L$ is arbitrary. It's inverse $\Psi$ is given by $\Psi(\psi)(l) = \sum_{i=-\infty}^\infty \psi(l t^i) t^{-i}$. We leave it to the reader to verify that these isomorphisms are well-defined and that they are functorial in the first entry, i.e.\ that given a map $L_0 \rightarrow L_1$ of right $R_\alpha[t\inv,t]$-modules, the diagram 
    \[
        \begin{tikzcd}
            \Hom_{R_\alpha[t\inv,t]}(L_1, M_\alpha\llbracket t\inv,t \rrbracket) \arrow[r]\arrow[d] & \Hom_{R_\alpha[t\inv,t]}(L_0, M_\alpha\llbracket t\inv,t \rrbracket)\arrow[d] \\
            \Hom_R(\iota^! L_1,M) \arrow[r] & \Hom_R(\iota^! L_0,M)
        \end{tikzcd}
    \]
    commutes. The claimed isomorphism is then an immediate consequence of the isomorphism $\Phi$. \qedhere
\end{proof}

The power of \cref{obs:SES} is that it relates the (co)homology of $C_\bullet$ with that of $\iota^! C_\bullet$ via the following immediate consequence of Shapiro's Lemma.

\begin{cor}\label{cor:LES}
    Let $C_\bullet$ be a chain complex of projective $R_\alpha[t\inv,t]$-modules. For any right $R$-module $M$, the short exact sequence of \cref{obs:SES} induces the following long exact sequences in homology and cohomology:
    \[
        \begin{tikzcd}[column sep=tiny]
			& \cdots \arrow[r] \arrow[d, phantom, ""{coordinate, name=ZZ}] 
			& \H_i(C_\bullet;M_\alpha\llbracket t\inv,t] \oplus M_\alpha[t\inv,t\rrbracket) \arrow[r] 
			& \H_i(C_\bullet; M_\alpha\llbracket t\inv,t \rrbracket) \arrow[dll,  rounded corners, to path={ -- ([xshift=2ex]\tikztostart.east)|- (ZZ) [near end]\tikztonodes-| ([xshift=-2ex]\tikztotarget.west)-- (\tikztotarget)}] \\
			& \H_{i-1}(\iota^! C_\bullet;M) \arrow[r] & \H_{i-1}(C_\bullet;M_\alpha\llbracket t\inv,t] \oplus M_\alpha[t\inv,t\rrbracket) \arrow[r] & \cdots
		\end{tikzcd}
    \]
    \[
        \begin{tikzcd}[column sep=tiny]
			& \cdots \arrow[r] \arrow[d, phantom, ""{coordinate, name=ZZ}] 
			& \H^i(C_\bullet;M_\alpha\llbracket t\inv,t] \oplus M_\alpha[t\inv,t\rrbracket) \arrow[r] 
			& \H^i(\iota^! C_\bullet; M) \arrow[dll,  rounded corners, to path={ -- ([xshift=2ex]\tikztostart.east)|- (ZZ) [near end]\tikztonodes-| ([xshift=-2ex]\tikztotarget.west)-- (\tikztotarget)}] \\
			& \H^{i+1}(C_\bullet;M_\alpha[t\inv,t]) \arrow[r] & \H^{i+1}(C_\bullet;M_\alpha\llbracket t\inv,t] \oplus M_\alpha[t\inv,t\rrbracket) \arrow[r] & \cdots
		\end{tikzcd}
    \]
\end{cor}

Before establishing some of the main results, we need the following key lemma.

\begin{lem}\label{lem:countable}
    Any non-zero $R_\alpha[t\inv,t \rrbracket$-module is uncountable.
\end{lem}
\begin{proof}
    Let $S \subset R_\alpha[t\inv,\rrbracket$ be the set of non-zero expressions of the form $\sum_{i=0}^\infty \varepsilon_i t^i$, where $\varepsilon_i \in \{0,1\}$ for each $i \geqslant 0$. Observe that if $s_0 \neq s_1$ are distinct elements of $S$, then $s_1 - s_0$ is a unit of $R_\alpha[t\inv,t\rrbracket$. This is because $s_1 - s_0 = (1 - p)t^{-k}$ for some $k \in \Z$ and some element $p$ of the form $p = \sum_{i = 1}^\infty \delta_i t^i$. Hence $(s_1 - s_0)\inv = t^k(1 + p + p^2 + \cdots)$.

    Let $M$ be a non-zero module over $R_\alpha[t\inv,t\rrbracket$ and let $m \in M \smallsetminus \{0\}$. We claim that the map $S \rightarrow M$ sending $s$ to $m\cdot s$ is injective. Suppose $m \cdot s_0 = m \cdot s_1$ for some $s_0,s_1 \in S$. If $s_0 \neq s_1$, then $m = m(s_1 - s_0)(s_1 - s_0)\inv = 0$, which is a contradiction. Hence, the map is injective, and since $S$ is uncountable, $M$ is too. \qedhere
\end{proof}

We are now ready to give a proof of Ranicki's Criterion based on the long exact sequences of \cref{cor:LES}. The result differs from Ranicki's in two ways: first, it is in terms of vanishing Novikov cohomology (instead of homology) and second, it does not require that the chain complex be bounded above or below. Note also that Ranicki's original proof was for chain complexes of free modules and with trivial twisting $\alpha$.

\begin{thm}[Ranicki's Criterion, cohomology version]\label{thm:ranicki}
    Let $R$ be a countable ring and let $C_\bullet$ be a chain complex of finitely generated projective $R_\alpha[t\inv,t]$-modules. For any integer $n \geqslant 0$, the following are equivalent:
    \begin{enumerate}
        \item\label{item:finite_domination} $\iota^! C_\bullet$ is finitely dominated;
        \item\label{item:conov_zero} $\H^i(C_\bullet; R_\alpha\llbracket t\inv, t]) = \H^i(C_\bullet; R_\alpha[t\inv, t\rrbracket) = 0$ for all $i$.
    \end{enumerate}
\end{thm}
\begin{proof}
    Assume that \ref{item:finite_domination} holds. The long exact sequence in cohomology of \cref{cor:LES} yields exact sequences 
    \[
        \H^i(C_\bullet; R_\alpha[t\inv,t]) \longrightarrow \H^i(C_\bullet; R_\alpha\llbracket t\inv,t] \oplus R_\alpha[t\inv,t\rrbracket) \longrightarrow \H^i(\iota^! C_\bullet; R)
    \]
    for all $i \geqslant 0$. Since $C_\bullet$ is a chain complex of finitely generated $R_\alpha[t\inv,t]$-modules and $R_\alpha[t\inv,t]$ is countable, it follows that $\H^i(C_\bullet; R_\alpha[t\inv,t])$ is countable. Moreover, since $\iota^! C_\bullet$ is chain homotopy equivalent to a chain complex that is finitely generated in degrees at most $n$, we also have that $\H^i(\iota^! C_\bullet; R)$ is countable. Since the sequence above is exact, we conclude that $\H^i(C_\bullet; R_\alpha\llbracket t\inv,t] \oplus R_\alpha[t\inv,t\rrbracket)$ is countable, and therefore vanishes by \cref{lem:countable}.

    \smallskip

    Now we assume that \ref{item:conov_zero} holds. Let $M$ be an arbitrary right $R$-module. The long exact sequence in cohomology of \cref{cor:LES} yields isomorphisms
    \[
        \H^i(\iota^! C_\bullet; M) \ \cong \ \H^{i+1}(C_\bullet; M_\alpha[t\inv,t])
    \]
    for all $i \geqslant 0$, by \cref{lem:acyclic_ring_module}. Let $\{M_j\}_{j \in J}$ be any directed system of $R$-modules such that $\varinjlim M_j = 0$. Then
    \[
        \varinjlim \left( (M_j)_\alpha[t\inv,t] \right) \ \cong \ \varinjlim \left( (M_j)_\alpha \otimes_R R_\alpha[t\inv,t] \right) \ = \ 0
    \]
    because tensor products commute with direct limits. Since $C_\bullet$ is a complex of finitely generated projective modules, we have
    \begin{align*}
        \varinjlim \H^i(\iota^! C_\bullet; M_j) \ &\cong \ \varinjlim \H^{i+1}(C_\bullet; (M_j)_\alpha[t\inv,t]) \\
        &\cong \ \H^{i+1}(C_\bullet; \varinjlim \left( (M_j)_\alpha[t\inv,t] \right)) \ = \ 0.
    \end{align*}
    By \cref{thm:Brown_crit_finiteness}, $\iota^! C_\bullet$ is finitely dominated. \qedhere
\end{proof}

\begin{rem}
    Let $S$ be an $R_\alpha[t\inv,t]$-algebra. If $C_\bullet$ is a bounded chain complex (meaning $C_i = 0$ for $|i|$ sufficiently large), then $\H^i(C_\bullet;S) = 0$ for all $i$ if and only if $\H_i(C_\bullet;S) = 0$ for all $i$ (this is an easy consequence of \cite[Lemma 2]{Brown_HomCritFinite}). Hence, the proof of \cref{thm:ranicki} gives a short new proof of Ranicki's Criterion.
\end{rem}

We decided to give the proof of \cref{thm:ranicki} in the case that $R$ is countable, since it is more straightforward and this covers most cases of interest (such as when $R = k[N]$, where $k$ is a countable ring and $N = \ker(G \rightarrow \Z)$ for some countable group $G$). As explained below, the case of a general ring $R$ follows from the countable case. Note that only the proof of the implication \ref{item:finite_domination} $\Rightarrow$ \ref{item:conov_zero} used the countability assumption.

\begin{proof}[Proof (of \cref{thm:ranicki}, for a general ring $R$)]
    We first set up some general notation: if $F$ is a free $S$-module for some ring $S$ and with a fixed basis $B$, then let $F|_{S'}$ denote the free $S'$-submodule $\bigoplus_{b \in B} S'b \leqslant F$.

    By assumption, there is a chain complex $D_\bullet$ of finitely generated projective $R$-modules chain homotopy equivalent to $\iota^! C_\bullet$. We may assume that $C_\bullet$ and $D_\bullet$ are complexes of finitely generated free modules, at the potential cost of them being nonzero in infinitely many degrees. Fix bases for all the modules in $C_\bullet$ and $D_\bullet$. Since $R_\alpha[t\inv,t]$ is free as an $R$-module with basis given by the powers of $t$, this choice induces a basis for $\iota^! C_\bullet$ as a free $R$-module. Note that all the bases are countable.

    We consider the following module maps:
    \begin{itemize}
        \item the boundary maps of the chain complexes $\iota^! C_\bullet$ and $D_\bullet$;
        \item a homotopy equivalence $h_\bullet \colon D_\bullet \rightarrow \iota^! C_\bullet$ with its chain homotopy inverse $g_\bullet \colon \iota^! C_\bullet \rightarrow D_\bullet$;
        \item a chain map $s_\bullet \colon D_\bullet \rightarrow D_{\bullet + 1}$ witnessing the fact that $g_\bullet h_\bullet$ is chain homotopic to the identity (i.e.\ a map satisfying $\id_{D_i} - g_i h_i = s_{i-1} \partial_i + s_{i+1} \partial_i$) as well as a chain map $r_\bullet \colon \iota^! C_\bullet \rightarrow \iota^! C_{\bullet+1}$ witnessing the fact that $h_\bullet g_\bullet$ is chain homotopic to the identity.
    \end{itemize}
    
    Let $E$ be the union of the $R$-bases of all the modules $\iota^! C_i$ and $D_i$. Let $R' \leqslant R$ be a countable subring such that $f(e) \in M|_{R'}$, where $e \in E$ is arbitrary and $f$ is any one of the maps above whose domain contains $e$ and whose codomain is $M$. Then $\iota^! C_\bullet|_{R'}$ and $D_\bullet|_{R'}$ are well defined chain complexes of free $R'$-modules, and the restrictions of the maps above witness the fact that they are chain homotopic.

    We will now prove that $\H^i(C_\bullet; R_\alpha[t\inv,t\rrbracket) = 0$ for all $i \leqslant n$. Let 
    \[
        z \colon C_i \longrightarrow R_\alpha[t\inv,t\rrbracket
    \]
    be a cocycle for some $i \leqslant n$, and suppose that $\{e_1, \dots, e_m\}$ is the fixed finite generating set of $C_i$ as an $R_\alpha[t\inv, t]$-module. By potentially enlarging the subring $R'$ from the previous paragraph, we may assume it is countable and contains all the coefficients of each element $z(e_j)$. We may moreover assume that $R'$ is $\alpha$-invariant (and countable) by replacing it with the subring generated by $\bigcup_{i\in\Z} \alpha^i(R')$. Then $z$ restricts to a cochain
    \[
        z' \colon C_i|_{R'} \longrightarrow R'_\alpha[t\inv,t].
    \]
    By \cref{thm:ranicki}, we have $\H^i(C_\bullet|_{R'}; R_\alpha'[t\inv,t\rrbracket) = 0$, so there exists a cochain $c' \colon C_{i-1}|_{R'} \rightarrow R_\alpha'[t\inv,t\rrbracket$ such that $z' = c' \partial$. But $c'$ defines an $R$-linear map 
    \[
        c \colon C_{i-1} \longrightarrow R_\alpha[t\inv,t\rrbracket
    \]
    by setting $c(e) = c'(e)$ on all basis elements $e$. We then also have $z = c \partial$, proving that $\H^i(C_\bullet; R_\alpha[t\inv,t\rrbracket) = 0$, as desired. \qedhere
\end{proof}

We will need the next two propositions later on. The first is a refinement of Ranicki's Criterion, which is often called Sikorav's Theorem, and has already been obtained in this form by Hillman and Kochloukova \cite[Theorem 5]{HillmanKochloukova_PDnCovers}. We will give a new proof the direction \ref{item:nov_zero_range} $\Rightarrow$ \ref{item:homotopy_finite_range}. Together with Hillman and Kochloukova's proof of the implication \ref{item:homotopy_finite_range} $\Rightarrow$ \ref{item:nov_zero_range}, this gives a proof of Sikorav's Theorem that does not involve analysing individual Novikov cycles in the complex of Novikov chains, and instead uses standard methods from homological algebra (such as mapping cones of chain complexes and long exact sequences in homology).

\begin{prop}[Sikorav's Theorem]\label{prop:Sikorav}
    Let $C_\bullet$ be a chain complex of projective $R_\alpha[t\inv,t]$-modules such that $C_i$ is finitely generated for all $i \leqslant n$ and $C_i = 0$ for all $i < 0$. The following are equivalent:
    \begin{enumerate}
        \item\label{item:homotopy_finite_range} $\iota^! C_\bullet$ is of finite $n$-type;
        \item\label{item:nov_zero_range} $\H_i(C_\bullet; R_\alpha\llbracket t\inv, t]) = \H_i(C_\bullet; R_\alpha[t\inv, t\rrbracket) = 0$ for all $i \leqslant n$.
    \end{enumerate}
\end{prop}
\begin{proof}
    Assume that \ref{item:nov_zero_range} holds. There is a finitely generated direct summand $C_{n+1}'$ of $C_{n+1}$ such that the modified chain complex
    \[
        C_\bullet' = \{0 \longrightarrow C_{n+1}' \longrightarrow C_n \longrightarrow \cdots \longrightarrow C_0 \longrightarrow 0\}
    \]
    still has $\H_n(C_\bullet'; R_\alpha\llbracket t\inv, t]) = \H_n(C_\bullet'; R_\alpha[t\inv, t\rrbracket) = 0$. This is because $R_\alpha\llbracket t\inv, t]$ is a ring, and therefore the module $Z_i(R_\alpha\llbracket t\inv, t] \otimes_{R_\alpha[t\inv,t]} C_\bullet)$ of cycles degrees $i \leqslant n$ is a direct summand of $R_\alpha\llbracket t\inv, t] \otimes_{R_\alpha[t\inv,v]} C_i$ for all $i \leqslant n$. A similar remark applies to $R_\alpha[t\inv, t\rrbracket$.
    
    Now, 
    \[
        \H_i(C_\bullet'; M_\alpha\llbracket t\inv, t]) \ = \ \H_i(C_\bullet'; M_\alpha[t\inv, t\rrbracket) = 0
    \]
    for any right $R$-module $M$ by \cref{lem:acyclic_ring_module} and all $i \leqslant n$. The long exact sequence in homology of \cref{cor:LES} yields isomorphisms
    \[
        \H_{i+1}(C_\bullet'; M_\alpha\llbracket t\inv, t\rrbracket) \ \cong \ \H_i(\iota^! C_\bullet; M)
    \]
    for all $i < n$. Note that $\prod R_\alpha \llbracket t\inv, t \rrbracket \cong \left(\prod R \right)_\alpha \llbracket t\inv, t \rrbracket$, where the product is taken over an arbitrary index set. Since $C_{i+1}$ is finitely generated for $i < n$, the homology of $C_\bullet$ commutes with products, and therefore
    \begin{align*}
        \H_i(\iota^! C_\bullet; \prod R) \ = \ \H_i(\iota^! C_\bullet'; \prod R) \ &\cong \ \H_{i+1}(C_\bullet'; (\prod R)_\alpha\llbracket t\inv, t\rrbracket) \\
        &\cong \ \H_{i+1}(C_\bullet'; \prod R_\alpha \llbracket t\inv, t \rrbracket) \\
        &\cong \ \prod \H_{i+1}(C_\bullet'; R_\alpha\llbracket t\inv,t \rrbracket) \\
        &\cong \ \prod \H_i(\iota^! C_\bullet'; R) \ = \ \prod \H_i(\iota^! C_\bullet; R)
    \end{align*}
    for $i < n$.
    
    We have a commuting diagram
    \[
        \begin{tikzcd}
            \H_{n+1}(C_\bullet'; \prod R_\alpha \llbracket t\inv, t \rrbracket) \arrow[r]\arrow[d, "\cong"] & \H_n(\iota^! C_\bullet'; \prod R) \arrow[r]\arrow[d] & 0 \\
            \prod \H_{n+1}(C_\bullet'; R_\alpha \llbracket t\inv, t \rrbracket)\arrow[r] & \prod \H_n(\iota^! C_\bullet';R) \arrow[r] & 0
        \end{tikzcd}
    \]
    with exact rows. Hence, $\H_n(\iota^! C_\bullet'; \prod R) \rightarrow \prod \H_n(\iota^! C_\bullet';R)$ is an epimorphism. But there is another commuting diagram 
    \[
        \begin{tikzcd}
            \H_n(\iota^! C_\bullet'; \prod R) \arrow[r]\arrow[d] & \H_n(\iota^! C_\bullet; \prod R) \arrow[r]\arrow[d] & 0 \\
            \prod \H_n(\iota^! C_\bullet';R) \arrow[r] & \prod \H_n(\iota^! C_\bullet; \prod R) \arrow[r] & 0
        \end{tikzcd}
    \]
    with exact rows. Hence $\prod \H_n(\iota^! C_\bullet';R) \rightarrow \H_n(\iota^! C_\bullet; \prod R)$ is an epimorphism, and hence \ref{item:homotopy_finite_range} follows from \cref{thm:Brown_crit_finiteness}. \qedhere
\end{proof}

The following result is the dual of the previous theorem and is proved similarly. It will be used in \cref{thm:co_abelian_cd}.

\begin{thm}\label{thm:cd_limits_commute}
    Let $C_\bullet$ be a chain complex of projective $R_\alpha[t\inv,t]$-modules with $C_i = 0$ for all $i > n$. Suppose that the functors $\H^i(C_\bullet;-)$ commute with direct limits of $R_\alpha[t\inv,t]$-modules for all $i \geqslant m$. If 
    \[
        \H^i(C_\bullet; R_\alpha\llbracket t\inv, t]) \ = \ \H^i(C_\bullet; R_\alpha[t\inv, t\rrbracket) \ = \ 0
    \]
    for all $i \geqslant m$, then $\H^i(\iota^! C_\bullet; -)$ commutes with direct limits for all $i \geqslant m$. If $n \geqslant m$, then $\iota^! C_\bullet$ is of cohomological dimension at most $n-1$.
\end{thm}

\begin{rem}
    It is unclear to us whether the converse holds, i.e.\ whether $\H^i(\iota^! C_\bullet; -)$ commuting with direct limits for $i \geqslant m$ implies vanishing Novikov cohomology in degrees at least $m$. It holds if $m = n$ (see \cref{cor:cd_chain_complex}), and if we additionally assume that $C_\bullet$ is a chain complex of finitely generated projectives, then the converse also holds at $m = n-1$.
\end{rem}

\begin{proof}[Proof (of \cref{thm:cd_limits_commute})]
    The proof that $\H^i(\iota^! C_\bullet; -)$ commutes with direct limits is similar to the proof of \cref{prop:Sikorav}, where it was shown that $\H_i(\iota^! C_\bullet;-)$ commutes with direct products in a certain range. We leave the details to the reader.

    By \cref{lem:acyclic_ring_module},
    \[
        \H^i(C_\bullet; M_\alpha\llbracket t\inv, t]) \ = \ \H^i(C_\bullet; M_\alpha[t\inv, t\rrbracket) \ = \ 0
    \]
    for all $i \geqslant m$ and any right $R$-module $M$. The long exact sequence in cohomology of \cref{cor:LES} then yields $\H^n(\iota^! C_\bullet; M) = 0$. Since $M$ was arbitrary, this shows that $\iota^! C_\bullet$ has cohomological dimension at most $n-1$. \qedhere
\end{proof}

As a corollary, we find that the top-dimensional Novikov cohomology of $C_\bullet$ completely controls the cohomological dimension of $\iota^! C_\bullet$, proving \cref{introthm:cohom_drop_complex} from the introduction in a more general form.

\begin{cor}\label{cor:cd_chain_complex}
    Let $C_\bullet$ be a chain complex of projective $R_\alpha[t\inv,t]$-modules of cohomological dimension $n$ and such that $C_n$ is finitely generated. The following are equivalent:
    \begin{enumerate}
        \item\label{item:cohom_dim} $\iota^! C$ is of cohomological dimension $n-1$;
        \item\label{item:top_Nov_zero} $\H^n(C_\bullet; R_\alpha\llbracket t\inv,t]) = \H^n(C_\bullet; R_\alpha[t\inv,t\rrbracket) = 0$.
    \end{enumerate}
\end{cor}
\begin{proof}
    If $\iota^! C_\bullet$ is of cohomological dimension $n-1$, then the long exact sequence in cohomology of \cref{cor:LES} yields the exact sequence
    \[
        \H^n(C_\bullet; R_\alpha[t\inv,t]) \longrightarrow \H^n(C_\bullet; R_\alpha\llbracket t\inv,t] \oplus R_\alpha[t\inv,t\rrbracket) \longrightarrow 0.
    \]
    If $R$ is a countable ring, then $\H^n(C_\bullet; R_\alpha[t\inv,t])$ is countable (since $C_n$) is finitely generated. If not, we can use the fact that $C_n$ is finitely generated to reduce to the countable case, much in the same way as in the proof of \cref{thm:ranicki} in the general case. We sketch the argument below.

    We may assume that the modules $C_i$ are all free. The fact that $\iota^! C_\bullet$ is of cohomological dimension $n-1$ means that there is a map $\rho \colon \iota^! C_{n-1} \rightarrow \iota^! C_n$ such that $\rho \partial = \id_{\iota^!C_n}$. Since $C_n$ is finitely generated, there is a finitely generated direct summand $A$ of $C_{n-1}$ containing $\partial(C_n)$. Writing $C_{n-1} = A \oplus B$, we can take $\rho$ to be zero on $B$. Representing $\partial$ and $\rho$ as matrices over $R$, they contain countably many entries, and there is an $\alpha$-invariant countable subring $R'$ of $R$ containing all of these entries. By the argument above, 
    \[
        \H^n(C_\bullet; R_\alpha'\llbracket t\inv,t]) \ = \ \H^n(C_\bullet; R_\alpha'[t\inv,t\rrbracket) \ = \ 0.
    \]
    We conclude that $\H^n(C_\bullet; R_\alpha\llbracket t\inv,t]) = \H^n(C_\bullet; R_\alpha[t\inv,t\rrbracket) = 0$ in the same way as in the proof of the general case of \cref{thm:ranicki}. \qedhere
\end{proof}

\section{The \texorpdfstring{$\Sigma^*$}{Σ*}-invariant}\label{sec:Sigma_inv}

Let $G$ be a group and let $R$ be a ring such that $\cd_R(G) = n$. If the trivial $R[G]$-module $R$ admits a projective resolution $P_\bullet \rightarrow R \rightarrow 0$ such that $P_i$ is finitely generated for all $i > n-m$, then we will say that $G$ is of \emph{type $\FP_m^*(R)$}. In other words, the trivial module admits a resolution that is finitely generated in the top $m$ degrees. Note that the finiteness property $\FP^*_m(R)$ only makes sense for groups with finite cohomological dimension over $R$, and that if a group is of type $\FP(R)$ then it is of type $\FP^*_m(R)$ for all $m$.

\begin{defn}
    Let $G$ be of type $\FP^*_m(R)$. The $m$th \emph{$\Sigma^*$-invariant} of $G$ (over $R$) is
    \[
        \Sigma_m^*(G;R) \ = \ \{ \chi \in \H^1(G,\R) \smallsetminus \{0\} \ : \ \H^i(G;\widehat{R[G]}^\chi) = 0 \ \text{for all} \ i > n-m \}.
    \]
\end{defn}

In particular, $\Sigma_1^*(G;R)$ is the set of characters with respect to which the top-dimensional Novikov homology of $G$ vanishes. It follows immediately from the definition that $\Sigma_{m+1}^*(G;R) \subseteq \Sigma_m^*(G;R)$ (whenever $G$ is of type $\FP^*_{m+1}(R)$). 

\begin{prop}\label{prop:open}
    Let $G$ be a finitely generated group of type $\FP^*_m(R)$ for some ring $R$ and integer $m$. Then $\Sigma_m^*(G;R)$ is open in $\H^1(G;\R)$.
\end{prop}

The proposition will follow quickly from \cref{lem:open} below, which will be well known to experts. The proof given here relies on chain contractions, and the following simple fact: if $\chi \colon G \rightarrow \R$ is positive on the support of $x \in R[G]$, then there is an open neighbourhood $U$ of $\chi$ such that $1 - x$ is invertible in $\widehat{R[G]}^\psi$ for all $\psi \in U$. We first give a couple of definitions.

\begin{defn}[Truncation]\label{def:truncation}
    Let $\chi \colon G \rightarrow \R$ be a character. The \emph{truncation} of an element $\sum_{g \in G} r_g g \in \widehat{R[G]}^\chi$ at height $t \in \R$ is the element $\sum_{g \in G} \overline{r}_g g \in R[G]$, where $\overline{r}_g = r_g$ if $\chi(g) < t$, and $\overline{r}_g = 0$ otherwise.
\end{defn}

Denote by $R^G$ the set of all formal series of elements of $G$ with coefficients in $R$ (which can also be thought of as the set of all functions $G \rightarrow R$). For every character $\chi \colon G \to \R$, there is a natural inclusion $\widehat{R[G]}^\chi \subseteq R^G$.

\begin{defn}[Definability over Novikov rings]\label{def:definability}
    Let $\chi, \psi \colon G \rightarrow \R$ be two characters. Say that $x \in \widehat{R[G]}^\chi$ is \emph{definable over $\widehat{R[G]}^\psi$} if $x \in \widehat{R[G]}^\chi \cap \widehat{R[G]}^\psi$ when viewing both Novikov rings as subsets of $R^G$.
\end{defn}

\cref{def:truncation,def:definability} extend immediately to matrices over Novikov rings. Let $A \in \mathrm{Mat}(\widehat{R[G]}^\chi)$. The truncation of $A$ at height $t \in \R$ is obtained by truncating all entries of $A$ at height $t$, and we say that $A$ is definable over $\widehat{R[G]}^\psi$ if all its entries are definable over $\widehat{R[G]}^\psi$.

\begin{lem}\label{lem:open}
    Let $C_\bullet$ be a chain complex of finitely generated projective $R[G]$-modules such that $C_i = 0$ for $i < 0$ and $C_i$. For any integer $m$, the set of non-zero characters $\chi$ such that $\H_i(C_\bullet; \widehat{R[G]}^\chi) = 0$ for all $i \leqslant m$ is open in $\H^1(G;\R) \smallsetminus \{0\}$.
\end{lem}
\begin{proof}
    By taking direct sums, we may assume that the modules $C_i$ are finitely generated and free. Let $\chi$ be a character such that $\H_i(C_\bullet; \widehat{R[G]}^\chi) = 0$ for all $i \leqslant m$. Fixing bases for the modules $C_i$ induces representations of maps between modules $\widehat{R[G]}^{\chi} \otimes_{R[G]} C_i$ as matrices over $\widehat{R[G]}^\chi$.
    
    By induction on $i$, we will show that $\widehat{R[G]}^{\chi} \otimes_{R[G]} C_\bullet$ admits a partial chain contraction of length $i$ whose maps are defined over $\widehat{R[G]}^\psi$ for all $\psi$ in some open neighbourhood $U \subseteq \H^1(G;\R) \smallsetminus \{0\}$ of $\chi$ (recall that a partial chain contraction of length $l$ is a sequence of maps
    \[
        s_i \colon \widehat{R[G]}^\chi \otimes_{R[G]} C_i \longrightarrow \widehat{R[G]}^\chi \otimes_{R[G]} C_{i+1}
    \]
    such that $\id = s_{i-1} \partial_i + \partial_{i+1} s_i$ for all $i \leqslant l$).

    Set $s_{-1} = 0$, which is defined over every Novikov ring, so there is nothing to prove in the base case. Suppose that $m \geqslant i > -1$ and that we have found maps $s_j$ for all $j < i$ defining a length $i-1$ partial chain contraction of $\widehat{R[G]}^{\chi} \otimes_{R[G]} C_\bullet$, which is defined over $\widehat{R[G]}^{\psi}$ for all $\psi$ in some open neighbourhood $U$ of $\chi$. By the Novikov acyclicity assumption, there is a map $\sigma_i \colon \widehat{R[G]}^\chi \otimes_{R[G]} C_i \to \widehat{R[G]}^\chi \otimes_{R[G]} C_{i+1}$ extending the partial chain contraction to one of length $i$. We will now modify $\sigma_i$ so that it is defined over an open set of characters.
    
    Let $\overline{\sigma}_i$ be a truncation of $\sigma_i$ at a height sufficiently large so that every entry of the matrix $M = \partial_{i+1}(\sigma_i - \overline{\sigma}_i) $ has positive support under $\chi$. We then have
    \[
        M \ = \ \id - \sigma_{i-1} \partial_i - \partial_{i+1} \overline{\sigma}_i,
    \]
    so $M$ is defined over $\widehat{R[G]}^\psi$ for $\psi$ in some open neighbourhood $V$ of $\chi$. Moreover, since $M$ has only finitely many non-zero entries, we may shrink $V$ so that the entries of $M$ are positively supported with respect to all $\psi \in V$. It follows that $\id - M$ is invertible with inverse $\sum_{i=0}^\infty M^i$ defined over $\widehat{R[G]}^\psi$ for all $\psi \in V$. Let $s_i = \overline{\sigma}_i (\id - M)\inv$. One easily checks that this extends the partial chain contraction which is defined in the open neighbourhood $U \cap V$ of $\chi$. This completes the induction.
    
    The existence of a length $m$ partial chain contraction defined over $\widehat{R[G]}^\psi$ for all $\psi$ in an open neighbourhood of $\chi$ implies that $\H_i(C_\bullet;\widehat{R[G]}^\psi) = 0$ for all $i \leqslant m$ and all $\psi$ in said neighbourhood. \qedhere
\end{proof}

\begin{rem}
    The proof above would be shorter if definability over the Novikov ring was an open condition on characters. This is not the case. Indeed, it is easy to construct examples of elements in $\widehat{R[G]}^\chi$ (with $G = \Z^2$, for instance) that are not defined over any $\widehat{R[G]}^\psi$ with $\psi \neq \chi$.
\end{rem}

\begin{proof}[Proof (of \cref{prop:open})]
    Let $P_\bullet \rightarrow R \rightarrow 0$ be a length $n$ resolution witnessing the $\FP^*_m(R)$ property for $G$. Let $P_{n-m}'$ be a finitely generated free factor of $P_{n-m}$ containing the image of $P_{n-m+1}$ under the boundary map. The complex 
    \[
        0 \longrightarrow P_n \longrightarrow \cdots \longrightarrow P_{n-m+1} \longrightarrow P_{n-m}'
    \]
    has vanishing cohomology (with any coefficients) in degrees above $n-m$ if and only if $P_\bullet$ does. Hence, we will work with this modified complex, and continue to denote it by $P_\bullet$. The cohomology $\H^i(P_\bullet; \widehat{R[G]}^\chi)$ in degrees $i > n-m$ is the cohomology of the cochain complex $\Hom_{R[G]}(P_i, \widehat{R[G]}^\chi) \cong \Hom_{R[G]}(P_i, R[G]) \otimes_{R[G]} \widehat{R[G]}^\chi$. Since $P_i$ is finitely generated, $\Hom_{R[G]}(P_i, R[G])$ is projective and finitely generated as an $R[G]$-module. The result then immediately follows from \cref{lem:open}. \qedhere
\end{proof}

As we saw in \cref{ex:Novikov}, the Novikov ring associated to an integral character $\chi \colon G \rightarrow \Z$ is a special case of the construction $R_\alpha[t\inv,t\rrbracket$. Thus, the results of the previous section readily apply to give theorems about group cohomology. The first such result is an immediate consequence of \cref{cor:cd_chain_complex}, and establishes the converse to \cite[Theorem 3.5]{Fisher_freebyZ} in the case of integral characters.

\begin{thm}\label{thm:if_and_only_if}
    Let $G$ be a group and let $R$ be a ring such that $\cd_R(G) < \infty$ and $G$ is of type $\FP^*_1(R)$. Let $\chi \colon G \rightarrow \Z$ be an integral character. Then $\cd_R(\ker \chi) = \cd_R(G)-1$ if and only if $\pm \chi \in \Sigma_1^*(G)$.
\end{thm}

\cref{introcor:open} follows at once from \cref{thm:if_and_only_if} and \cref{prop:open}.

\begin{cor}\label{cor:open}
    Let $G$ be a group of type $\FP(R)$ for some ring $R$. The set of characters
    \[
        \{ \chi \colon G \longrightarrow \Q \ : \ \cd_R(\ker \chi) = \cd_R(G) - 1 \}
    \]
    is open in $\H^1(G;\Q)$.
\end{cor}

The remainder of the section will be devoted to proving \cref{introthm:coabelian}. We first need a result about the finiteness properties of kernels of maps to poly-$\Z$ groups.

\begin{notation}\label{not:polyZ}
    Recall that a group $P$ is \emph{poly-$\Z$} if it admits a subnormal series
    \[
        \{1\} = P_d \leqslant P_{d-1} \leqslant \cdots \leqslant P_1 \leqslant P_0 = P
    \]
    such that $P_i/P_{i+1} \cong \Z$ for all $0 \leqslant i < d$. The integer $d$ is the \emph{length} of the poly-$\Z$ group $P$; it coincides with its cohomological dimension, and therefore is uniquely determined. We will always assume that our poly-$\Z$ groups come with a fixed subnormal series as above.

    For an epimorphism $\chi \colon G \rightarrow P$ with $P$ poly-$\Z$, there is a distinguished sequence of integral characters 
    \[
        \chi_i \colon G_i \longrightarrow G_i/G_{i+1} \cong \Z
    \]
    where $G_i := \chi\inv(P_i)$. Note that if $P$ is of length $d$, then $G_d = \ker \chi$.
\end{notation}

The following two results are easy consequences of \cref{prop:Sikorav} and \cref{thm:cd_limits_commute}. We will need them when studying We use the following notation: if $M$ is an $R[G]$-module and $N \leqslant G$ is a pair of groups, then the $R[N]$-module obtained by restricting $M$ will be denoted by $\res_N^G M$. We follow \cref{not:polyZ} in the statement of the next proposition.

\begin{prop}\label{prop:finite_polyZ}
    Let $\chi \colon G \rightarrow P$ be an epimorphism, where $P$ is a poly-$\Z$ group of length $d$. Let $C_\bullet$ be a chain complex of projective $R[G]$-modules of finite $n$-type. The following are equivalent:
    \begin{enumerate}
        \item\label{item:iterated_finite} $\res_{G_d}^G C_\bullet$ is of finite $n$-type;
        \item\label{item:iterated_vanish} $\H_i(\res_{G_j}^G C_\bullet;\widehat{R[G_j]}^{\pm\chi_j}) = 0$ for all $0 \leqslant j < d$.
    \end{enumerate}
\end{prop}
\begin{proof}
    We prove the result by induction on $d$. If $d = 1$, then the statement of the proposition is identical to that of \cref{prop:Sikorav}.

    Suppose \ref{item:iterated_finite} holds. Then $\res_{G_1}^G C_\bullet$ is of finite $n$-type. Indeed, there is a convergent spectral sequence
    \[
        \H^p(G_1/G_d; \H^q(\res_{G_d}^G C_\bullet; M)) \ \Longrightarrow \ \H^{p+q}(\res_{G_1}^G C_\bullet; M)
    \]
    for any $R[G_1]$-module $M$. This is just the Lyndon--Hochschild--Serre spectral sequence, except one replaces a resolution of the trivial $R[G]$-module $R$ by the chain complex $\res_{G_d}^G C_\bullet$; the proof of its existence is identical (see \cite[Chapter VII]{BrownGroupCohomology}). By \cref{thm:Brown_crit_finiteness}, $\varinjlim \H^q(\res_{G_d}^G C_\bullet; M_j) = 0$ for all $q \leqslant n$ and all directed systems $\{M_j\}$ with vanishing direct limit. Moreover, $\H^p(G_1/G_d; -)$ preserves direct limits in all degrees, since $G_1/G_d$ is a poly-$\Z$ group and therefore of type $\FP$. Thus, $\varinjlim \H^p(\res_{G_1}^G C_\bullet; M_j) = 0$ for all $p \leqslant n$ and all directed systems $\{M_j\}$ with vanishing direct limit. Again by \cref{thm:Brown_crit_finiteness}, $\res_{G_1}^G C_\bullet$ has the appropriate finiteness conditions. By \cref{prop:Sikorav}, we conclude that
    \[
        \H_j(\res^G_{G_1} C_\bullet ; \widehat{R[G]}^{\pm\chi_0}) \ = \ 0
    \]
    for all $j \leqslant n$. By induction, we also have
    \[
        \H_j(\res_{G_{i+1}}^{G} C_\bullet; \widehat{R[G_i]}^{\pm\chi_i}) \ \cong \ \H_j(\res_{G_{i+1}}^{G_1} C_\bullet; \widehat{R[G_i]}^{\pm\chi_i}) \ = \ 0
    \]
    for all $j \leqslant n$ and all $0 < i < d$.

    \smallskip

    If \ref{item:iterated_vanish} holds, then $\res_{G_1}^G C_\bullet$ is of finite $n$-type by \cref{prop:Sikorav}. By induction, $\res_{G_d}^{G_1} C_\bullet$ is of finite $n$-type. \qedhere
\end{proof}

\begin{prop}\label{prop:cd_polyZ}
    Let $\chi \colon G \rightarrow P$ be an epimorphism, where $P$ is a poly-$\Z$ group of length $d$. Let $C_\bullet$ be a chain complex of projective $R[G]$-modules such that $C_i = 0$ for all $i > n$ and such that $\H^i(G;-)$ commutes with direct limits of $R[G]$-modules for all $i > n-d$. If 
    \[
        \H^i(\res_{G_j}^G C_\bullet;\widehat{R[G_j]}^{\pm\chi_j}) \ = \ 0
    \]
    for all $i > n-d$ and all $0 \leqslant j < n$, then $\res_{G_d}^G C_\bullet$ is of cohomological dimension at most $n - d$.
\end{prop}
\begin{proof}
    We prove the result by induction on $d$. If $d = 1$, then it follows immediately from \cref{cor:cd_chain_complex}. Suppose that $d > 1$. Again by \cref{cor:cd_chain_complex}, $\res_{G_1}^G C_\bullet$ is chain homotopy equivalent to a chain complex of projective modules $D_\bullet$ which is of cohomological dimension at most $n-1$ and such that $\H^i(D_\bullet;-)$ commutes with direct limits for all $i > n-d$. The claim now follows by induction. \qedhere
\end{proof}

For the proof of \cref{introthm:coabelian}, we will need the following result of Farber, Geoghegan, and Sch\"{u}tz, which gives a version of Sikorav's Theorem for chain complexes.

\begin{thm}[{\cite[Theorem 8]{FarberGeogheganSchutz_SikoravComplexes}}]\label{thm:Sikorav_for_complexes}
    Let $C_\bullet$ be a chain complex of projective $R[G]$-modules such $C_i$ is finitely generated for $i \leqslant n$ and zero for $i < 0$. If $N \trianglelefteqslant G$ is such that $G/N$ is Abelian, then the following are equivalent:
    \begin{enumerate}
        \item $\res_N^G C_\bullet$ is of finite $n$-type;
        \item for all $\chi \in S(G,N)$, we have $\H_i(C_\bullet;\widehat{R[G]}^\chi) = 0$ for all $i \leqslant n$.
    \end{enumerate}
\end{thm}

\begin{rem}
    In \cite{FarberGeogheganSchutz_SikoravComplexes}, the result is stated for $R = \Z$. Just as in the case of Sikorav's Theorem for groups, the proof goes through to the case of a general unital associative ring $R$ without modification.
\end{rem}

We are now ready to prove \cref{introthm:coabelian} for chain complexes of projectives. Recall that if $N \leqslant G$ is a pair of groups and $M$ is a left $R[N]$-module, then the corresponding \emph{induced module} is $\Ind_N^G M = R[G] \otimes_{R[N]} M$, and is a left $R[G]$-module.

\begin{thm}\label{thm:co_abelian_cd}
    Let $C_\bullet$ be a chain complex of finitely generated projective $R[G]$-modules such that $C_i = 0$ for all $i > n$, and let $\chi \colon G \rightarrow \R$ be a character with free Abelian image of rank $d$. If
    \[
        \H^i(C_\bullet; \widehat{R[G]}^\varphi) \ = \ 0
    \]
    for all $\varphi \in S(G,\ker \chi)$ and all $i > n-d$, then $\res_{\ker \chi}^G C_\bullet$ has cohomological dimension at most $n-d$.
\end{thm}
\begin{proof}
    Since the modules $C_i$ are finitely generated, 
    \[
        \Hom_{R[G]}(C_\bullet, \widehat{R[G]}^\varphi) \ \cong \ \Hom_{R[G]}(C_\bullet, R[G]) \otimes_{R[G]} \widehat{R[G]}^\varphi.
    \]
    The Novikov acyclicity assumption and a \cref{thm:Sikorav_for_complexes} imply that 
    \[
        \res_{\ker\chi}^G \Hom_{R[G]}(C_\bullet, R[G])
    \]
    is chain homotopy equivalent to a chain complex $D_\bullet$ of projective $R[\ker\chi]$-modules such that $D_i$ is finitely generated for all $i > n-d$.

    For the remainder of the proof, we follow \cref{not:polyZ}. Let $P = \chi(G) \cong \Z^d$ and fix a series of subgroups 
    \[
        \{1\} = P_d \leqslant P_{d-1} \leqslant \cdots \leqslant P_1 \leqslant P_0 = P
    \]
    such that $P_i/P_{i+1} \cong \Z$ for all $0 \leqslant i < d$. Let the maps $\chi_i \colon G_i = \chi\inv(P_i) \rightarrow P_i/P_{i+1}$ be the associated integral characters. By \cref{prop:finite_polyZ}, for all $0 \leqslant j < d$, the cochain complexes 
    \begin{multline*}
        \res_{G_j}^G \Hom_{R[G]}(C_\bullet, R[G]) \otimes_{R[G_j]} \widehat{R[G_j]}^{\pm\chi_j} \\ 
        \cong \ \Hom_{R[G]}(C_\bullet, R[G]) \otimes_{R[G]} \Ind_{G_j}^G \widehat{R[G_j]}^{\pm\chi_j} \\
        \cong \ \Hom_{R[G]}(C_\bullet, \Ind_{G_j}^G \widehat{R[G_j]}^{\pm\chi_j})
    \end{multline*}
    are acyclic in degrees greater than $n-d$, so we have $\H^i(C_\bullet; \Ind_{G_j}^G \widehat{R[G_j]}^{\pm\chi_j}) = 0$ for all $i > n-d$ and $0 \leqslant j < d$. The result will follow quickly from the following claim.

    \begin{claim}
        Let $\chi \colon G \rightarrow P$ be an epimorphism, where $P$ is a poly-$\Z$ group of length $d$. Let $F_\bullet$ be a chain complex of projective $R[G]$-modules such that $F_i = 0$ for all $i > n$ and such that $\H^i(F_\bullet;-)$ commutes with direct limits for all $i > n-d$. If 
        \[
            \H^i(F_\bullet; \Ind_{G_j}^G \widehat{R[G_j]}^{\pm\chi_j}) \ = \ 0
        \]
        for all $i > n-d$ and $0 \leqslant j < d$, then $\H^i(\res_{G_j}^G F_\bullet; \widehat{R[G_j]}^{\pm\chi_j}) = 0$ for all $i > n-d$ and $0 \leqslant j < d$.
    \end{claim}
    \begin{proof}
        We will prove this by induction on $d$. If $d = 1$, then there is nothing to show, so we assume that $d > 1$. Fix some $j < d$. Let $M = \Ind_{G_j}^{G_1} \widehat{R[G_j]}^{\pm\chi_j}$, and let $t \in G$ map to the generator of $G/G_1$ such that $\chi_0(t) = 1$. As in \cref{obs:SES}, there is a short exact sequence
        \[
            0 \longrightarrow \Ind_{G_1}^G M \longrightarrow M_\alpha\llbracket t\inv, t] \oplus M_\alpha[t\inv, t\rrbracket \longrightarrow \coInd_{G_1}^G M \longrightarrow 0,
        \]
        where $\alpha$ is induced by the conjugation action of $t$. We know
        \[
            \H^i(F_\bullet; \Ind_{G_1}^G M) \ = \ \H^i(F_\bullet; \Ind_{G_j}^G \widehat{R[G_j]}^{\pm\chi_j}) \ = \ 0
        \]
        for all $i > n-d$. We also know that 
        \[
            \H^i(F_\bullet;M_\alpha\llbracket t\inv, t]) \ = \ \H^i(F_\bullet;M_\alpha[t\inv, t\rrbracket) \ = \ 0
        \]
        by \cref{lem:acyclic_ring_module}, since $M_\alpha[t\inv, t\rrbracket$ is a right $\widehat{R[G]}^{\chi_0} = R[G_1]_\alpha[t\inv,t\rrbracket$-module and
        \[
            \H^i(F_\bullet;\widehat{R[G]}^{\chi_0}) \ = \ 0
        \]
        (and similarly for $M_\alpha\llbracket t\inv,t]$). It then follows that
        \[
            \H^i(F_\bullet; \coInd_{G_1}^G M) \ = \ \H^i(\res_{G_1}^G F_\bullet; \Ind_{G_j}^{G_1} \widehat{R[G_j]}^{\pm\chi_j}) \ = \ 0.
        \]
        for all $i > n-d$. Since $\H^i(\res_{G_1}^G F_\bullet; -)$ commutes with direct limits for all $i > n-d$ by \cref{prop:cd_polyZ}, the claim holds by induction. \renewcommand\qedsymbol{$\diamond$}\qedhere
    \end{proof}

    By the claim, $\H^i(\res_{G_j}^G C_\bullet; \widehat{R[G_j]}^{\pm\chi_j}) = 0$ for all $i > n-d$ and all $0 \leqslant j < d$. By \cref{prop:cd_polyZ}, $\res_{\ker \chi}^G C_\bullet$ is of cohomological dimension at most $n - d$. \qedhere
\end{proof}

\cref{introthm:coabelian} follows immediately by taking $C_\bullet$ to be a projective resolution of $G$ witnessing the $\FP_d^*(R)$ property.

\begin{cor}\label{cor:coabelian}
    Let $G$ be a group of type $\FP^*_d(R)$ with $\cd_R(G) = n$, and suppose that $N \trianglelefteqslant G$ is such that $G/N \cong \Z^d$. If $S(G,N) \subseteq \Sigma_d^*(G;R)$, then $\cd_R(N) = n - d$.
\end{cor}

While \cref{cor:coabelian} has a converse for $d = 1$, the converse fails in general, as the following example shows. The author thanks Grigori Avramidi for telling him that many closed hyperbolic $3$-manifolds are free-by-$\Z^2$.

\begin{ex}\label{ex:bad_3_mfld}
    Let $M$ be a closed fibred hyperbolic $3$-manifold with $b_1(M) \geqslant 2$, and let $\chi \colon \pi_1(M) \rightarrow \Z$ be an algebraic fibration induced by a fibration of $M$ over $S^1$ with fibre a hyperbolic surface $S_g$. By openness of the $\Sigma$-invariant and the fact that $b_1(M) \geqslant 2$, we may perturb $\chi$ to obtain a new algebraic fibration $\chi' \colon \pi_1(M) \rightarrow \Z$ not equal to $\pm\chi$. The kernel of the map $f \colon \pi_1(M) \rightarrow \Z^2$ given by $\alpha \mapsto (\chi(\alpha), \chi'(\alpha))$ is an infinite-index subgroup of $\ker \chi \cong \pi_1(S_g)$ and therefore is non-finitely generated free. Hence, $\pi_1(M)$ is free-by-$\Z^2$. Let $F = \ker f$. By \cref{introthm:BNS}, there is a character $\chi \in S(\pi_1(M),F)$ such that $\H_1(\pi_1(M); \widehat{\Z[\pi_1(M)]}^\chi) \neq 0$ (because $F$ is not finitely generated), so $\H^2(\pi_1(M); \widehat{\Z[\pi_1(M)]}^\chi) \neq 0$ by Poincar\'e-duality. But the converse of \cref{cor:coabelian} would imply that $\H^i(\pi_1(M);\widehat{\Z[\pi_1(M)]}^\chi) = 0$ for $i = 2,3$, since $\cd(F) = 1$, so it cannot hold.
\end{ex}

We conclude the section by proving \cref{introcor:PD}. A group $G$ is an $n$-dimensional \emph{Poincar\'e-duality group} over $R$ (or, more briefly, a $\mathrm{PD}_R^n$-group) if it satisfies
\[
    \H^i(G;M) \ \cong \ \H_{n-i}(G;M)
\]
for every $R[G]$-module $M$. Homology commutes with direct limits, which implies that the cohomology of a Poincar\'e-duality group commutes with direct limits. Hence, by \cref{thm:Brown_crit_finiteness}, $\mathrm{PD}_R^n$-groups are of type $\FP(R)$. 

\begin{cor}\label{cor:PD}
    Let $G$ be a $\mathrm{PD}_R^n$-group and let $\chi \colon G \rightarrow \Z^d$ be an epimorphism. If $\ker \chi$ is of type $\FP_{d-1}(R)$, then $\cd_R(G) = n-d$.
\end{cor}
\begin{proof}
    By the higher version of \cref{introthm:BNS} (see \cite{BieriRenzValutations}, or \cite[Theorem 5.3]{Fisher_Improved} for the case of an arbitrary ring $R$), $S(G,N) \subseteq \Sigma_{d-1}(G;R)$, where $N = \ker \chi$. By Poincar\'e-duality, $S(G,N) \subseteq \Sigma_d^*(G;R)$. By \cref{cor:coabelian}, $\cd_R(G) = n-d$. \qedhere
\end{proof}

\section{An application to RFRS groups}\label{sec:RFRS}

\begin{defn}
    A group $G$ is said to be \emph{residually finite rationally solvable}, or \emph{RFRS}, if 
    \begin{enumerate}
        \item there is a chain of finite-index normal subgroups $G = G_0 \geqslant G_1 \geqslant G_2 \geqslant \dots$ such that $\bigcap_{i \geqslant 0} G_i = \{1\}$, and
        \item $\ker(G_i \rightarrow \Q \otimes_\Z G_i/[G_i,G_i])$ is a subgroup of $G_{i+1}$ for all $i \geqslant 0$. 
    \end{enumerate}
    A chain satisfying these properties will be called a \emph{witnessing chain} for the RFRS group.
\end{defn}

\begin{defn}
    It is easy to see that finitely generated RFRS groups are residually (torsion-free solvable), and therefore satisfy the strong Atiyah conjecture by \cite{SchickL2Int2002}. By a result of Linnell \cite{LinnellDivRings93}, this implies that the division closure of $\Q[G]$ in the algebra of affiliated operators $\mathcal U(G)$ is a division ring, which we call $\mathcal{D}_{\Q[G]}$. The \emph{$n$th $L^2$-Betti numbers} of $G$ is
    \[
        \b{n}(G) \ := \ \dim_{\mathcal D_{\Q[G]}} \H^n(G; \mathcal D_{\Q[G]}),
    \]
    where $\dim_{\mathcal D_{\Q[G]}} M$ the rank of a (necessarily free of unique rank) $\mathcal D_{\Q[G]}$-module $M$. For more details, we refer the reader to \cite[Chapter 10]{Luck02}.

    In fact, for RFRS groups $G$, Jaikin-Zapirain shows that $k[G]$ embeds into a division ring $\Dk{G}$ for any field $k$, which shares many properties with $\mathcal D_{\Q[G]}$ \cite[Corollary 1.3]{JaikinZapirain2020THEUO}. We may use this division ring to define the $L^2$-Betti numbers of a RFRS group over a field $k$ by
    \[
        \b{n}(G;k) \ := \ \dim_{\Dk{G}} \H^n(G; \Dk{G}).
    \]
    In particular, $\b{n}(G;\Q) = \b{n}(G)$.
\end{defn}

For us, the important common feature of all the division rings $\Dk{G}$ (as the ground field $k$ varies) is that they can be expressed as directed unions of rings which are closely related to the Novikov rings (over $k$) of $G$ and its finite-index subgroups. For $k = \Q$, this is the crux of Kielak's fibring theorem \cite{KielakRFRS}. This was extended by Jaikin-Zapirain to all fields $k$ in the appendix of \cite{JaikinZapirain2020THEUO}. We give a rough sketch of how this works, emphasising the details that will be necessary for us.

Fix a finitely generated RFRS group $G$ and a witnessing chain $(G_i)_{i\geqslant 0}$. In \cite[Definition 4.3]{KielakRFRS}, Kielak introduces the notion of a \emph{rich set} of characters in $\H^1(G_i;\R) \smallsetminus \{0\}$, a notion which depends the chosen witnessing chain $(G_i)_{i \geqslant 0}$. The definition of a rich set is somewhat technical, but we will only need the following properties:
\begin{itemize}
    \item rich sets are open;
    \item rich sets are nonempty;
    \item the intersection of rich sets is rich \cite[Lemma 4.4]{KielakRFRS}.
\end{itemize}

\begin{rem}
    In \cite[Definition 4.3]{KielakRFRS}, it is not required that rich sets be open. This is however simply a misprint, and indeed it is important that rich sets be open for the proofs of \cite[Lemma 4.4]{KielakRFRS} and \cite[Lemma 4.13]{KielakRFRS}.
\end{rem}

If $U \subseteq \H^1(G_i;\R)$ is any set, then Kielak defines \cite[Definition 3.15]{KielakRFRS} an associated subring $\mathcal D_{k[G_i],U} \subseteq \Dk{G}$ (see also \cite[Section 5.2]{JaikinZapirain2020THEUO}). Again the definition of $\mathcal D_{k[G_i],U}$ is somewhat technical, but we will really only use some basic properties (references are provided for the interested reader):
\begin{itemize}
    \item $k[G_i]$ is contained in $\mathcal D_{k[G_i],U}$ for every choice of $U$ (follows directly from \cite[Definition 3.15]{KielakRFRS});
    \item if $U \subseteq V \subseteq \H^1(G_i;\R)$, then $\mathcal D_{k[G_i],V} \subseteq \mathcal D_{k[G_i],U}$ (follows directly from \cite[Definition 3.15]{KielakRFRS});
    \item for any choice of $U$, the conjugation action of $G$ on $k[G_i]$ extends to an action on $\mathcal D_{k[G_i],U}$, and we may thus form the twisted group ring 
    \[
        \mathcal D_{k[G_i],U} * G/G_i \ \cong \ \mathcal D_{k[G_i],U} \otimes_{k[G_i]} k[G],
    \]
    and there is a natural embedding $\mathcal D_{k[G_i],U} * G/G_i \subseteq \Dk{G}$ (see \cite[Proposition 5.3]{JaikinZapirain2020THEUO});
    \item for every $\chi \in U$, there is a map $\mathcal D_{k[G_i],U} \rightarrow \widehat{k[G_i]}^\chi$, or in other words, $\widehat{k[G_i]}^\chi$ is a $\mathcal D_{k[G_i],U}$-algebra (see \cite[Lemma 3.13]{KielakRFRS}).
\end{itemize}

We continue with the notation above: $G$ is a RFRS group with witnessing chain $(G_i)_{i\geqslant 0}$. Let $n$ be a non-negative integer and let $\mathbf{U} = (U_n, U_{n+1}, \dots)$, where $U_i \leqslant \H^1(G_i;\R)$ is a rich set for each $i \geqslant n$. We call such a $\mathbf U$ a \emph{sequence of rich sets}. Define
\[
    \mathcal{D}_{k[G], \mathbf U} \ = \ \bigcap_{i \geqslant n} \mathcal{D}_{k[G_i],U_i} * G/G_i.
\]
The main technical theorem of Kielak is that $\Dk{G}$ is covered by the rings $\mathcal{D}_{k[G], \mathbf U}$.

\begin{thm}[{\cite[Theorem 4.13]{KielakRFRS}}]
    $\Dk{G} = \bigcup_{\mathbf U} \mathcal{D}_{k[G], \mathbf U}$, where the union is taken over all sequences of rich sets $\mathbf U$.
\end{thm}

We only need the following additional fact.

\begin{lem}\label{lem:directed_system}
    The rings $\{\mathcal{D}_{k[G], \mathbf U}\}_{\mathbf U}$ form a directed system under inclusion.
\end{lem}
\begin{proof}
    Let $m$ and $n$ be non-negative integers, and let $\mathbf U = (U_m, U_{n+1}, \dots)$ and $\mathbf V = (V_n, V_{n+1}, \dots)$ be sequences of rich sets. Let $M = \max\{m,n\}$. Then $\mathcal{D}_{k[G], \mathbf U}$ and $\mathcal{D}_{k[G], \mathbf V}$ are each contained in $\mathcal{D}_{k[G], \mathbf U \cap \mathbf V}$, where $\mathbf U \cap \mathbf V := (U_M \cap V_M, U_{M+1} \cap V_{M+1}, \dots)$. By the facts listed above, $\mathbf U \cap \mathbf V$ is a sequence of rich sets, and 
    \[
        \mathcal{D}_{k[G], \mathbf U} \ \subseteq \ \mathcal{D}_{k[G], \mathbf U \cap \mathbf V} \ \supseteq \ \mathcal{D}_{k[G], \mathbf V}.
    \]
    This proves that the system of rings $\mathcal{D}_{k[G], \mathbf U}$ is directed. \qedhere
\end{proof}

\begin{lem}\label{lem:direct_limit}
    Let $G$ be a group with $\cd_k(G) = n$ for some field $k$. Let $S = \varinjlim S_j$, where $\{S_j\}_{j \in J}$ is a directed system of $k[G]$-algebras. Suppose that $\H^i(G;-)$ commutes with direct limits of $k[G]$-modules for all $i \geqslant m$. If $\H^i(G;S) = 0$ for all $i \geqslant m$, then there is some $j \in J$ such that  $\H^i(G;S_j) = 0$ for all $i \geqslant m$.
\end{lem}
\begin{proof}
    Suppose that, for a fixed $l \geqslant m$, there exists $j \in J$ with $\H^i(G;S_j) = 0$ for all $i > l$ (we trivially know this holds for $l = n$ and all $j \in J$). Then $\H^l(G;-)$ is a right exact functor on the category of $S_j$-modules. By \cref{thm:EW_thm}, there is a natural isomorphism
    \[
        \H^l(G;-) \ \cong \ \H^l(G;S_j) \otimes_{S_j} -
    \]
    as functors on the category of $S_j$-modules.
    
    The functor $\H^l(G;-)$ commutes with direct products of $S_j$-modules, and therefore so does the functor $\H^l(G;S_j) \otimes_{S_j} -$. It is a well-known fact (and an easy exercise) that $\H^l(G;S_j)$ must be finitely generated as an $S_j$-module. Let $z_1, \dots, z_s$ be a finite generating set for $\H^l(G;S_j)$. Since
    \[
        \H^l(G;S) \ \cong \ \H^l(G;S_j) \otimes_{S_j} S \ \cong \ \varinjlim_{j' \geqslant j} \H^l(G;S_j) \otimes_{S_j} S_{j'} \ = \ 0,
    \]
    we see that there is some $j' \in J$ such that $z_i \otimes 1 = 0$ in $\H^l(G;S_j) \otimes_{S_j} S_{j'}$ for each $i = 1, \dots, s$. But these elements also generate $\H^l(G;S_j) \otimes_{S_j} S_{j'}$, so we conclude that 
    \[
        \H^l(G;S_{j'}) \ \cong \ \H^l(G;S_j) \otimes_{S_j} S_{j'} \ = \ 0.
    \]
    Since $S_{j'}$ is an $S_j$-module, we also have $\H^i(G; S_{j'}) = 0$ for all $i > l$ by \cref{lem:acyclic_ring_module}. The lemma then follows by induction. \qedhere
\end{proof}

We can now prove the main result of this section, which implies \cref{introthm:RFRS} by taking $d = 1$ and using the Stallings--Swan Theorem \cite{Stallings_cd1,Swan_cd1}, which characterises free groups as the groups of cohomological dimension at most one.

\begin{thm}\label{thm:RFRS_cd_drop}
    Let $G$ be a RFRS group with $\cd_k(G) = n$ for some field $k$. Suppose $\H^i(G;-)$ commutes with direct limits of $k[G]$-modules for all $i > d$. The following are equivalent:
    \begin{enumerate}
        \item\label{item:weak_dim} $\Dk{G}$ is of weak dimension at most $d$ as a $k[G]$-module;
        \item\label{item:all_subgps} for every subgroup $H \leqslant G$, we have $\b{i}(H;k) = 0$ for all $i > d$;
        \item\label{item:subnormal} there is a subnormal series $\Gamma_l < \Gamma_{l-1} < \cdots < \Gamma_0 = G$, where each quotient $\Gamma_i/\Gamma_{i+1}$ is either finite or cyclic, and $\cd_k(\Gamma_l) = d$.
    \end{enumerate}
\end{thm}

The proof will show that in \ref{item:subnormal}, the quotients $\Gamma_i/\Gamma_{i+1}$ will alternate between being finite and infinite cyclic, and when $\Gamma_i/\Gamma_{i+1} \cong \Z$ we have $\cd_k(\Gamma_{i+1}) < \cd_k(\Gamma_i)$.

\begin{proof}[Proof (of \cref{thm:RFRS_cd_drop})]
    The implication \ref{item:weak_dim} $\Rightarrow$ \ref{item:all_subgps} follows quickly from the fact that
    \[
        \H_i(H;\Dk{G}) \ \cong \ \H_i(H;\Dk{H}) \otimes_{\Dk{H}} \Dk{G}
    \]
    for all $i$, by flatness of $\Dk{G}$ over $\Dk{H}$.

    \smallskip

    Suppose that \ref{item:all_subgps} holds. By \cref{lem:direct_limit,lem:directed_system}, there is a sequence of rich sets $\mathbf U = (U_s, U_{s+1}, \dots)$ such that $\H^i(G;\mathcal{D}_{k[G], \mathbf U}) = 0$ for all $i > d$. But $\mathcal{D}_{k[G], \mathbf U} \subseteq \mathcal D_{k[G_s],U_s} * G/G_s$, and therefore
    \[  
        0 \ = \ \H^i(G; \mathcal D_{k[G_s],U_s} * G/G_s) \ \cong \ \H^i(G_s; \mathcal D_{k[G_s],U_s}) 
    \]
    for all $i > d$ by \cref{lem:acyclic_ring_module}. We will take $\Gamma_1 = G_s$. Recall that $\widehat{k[G_s]}^\chi$ is a $\mathcal D_{k[G_s],U_s}$-algebra for every $\chi \in U_s$. Since $-U_s \cap U_s$ is also a rich (and hence open) set, there is an integral character $\chi \colon G_s \rightarrow \Z$ such that 
    \[
        \H^i(G_s; \widehat{k[G_s]}^{\pm\chi}) \ = \ 0
    \]
    for all $i > d$. By \cref{cor:cd_chain_complex}, $\cd_k(\ker \chi) = n-1$ and $\H^i(\ker \chi; -)$ commutes with direct limits for all $i > d$. We take $\Gamma_2 = \ker \chi$, and repeat this argument with $\Gamma_2$ instead of $G$ to extend the chain to $\Gamma_2 > \Gamma_3 > \Gamma_4$, where $\Gamma_2/\Gamma_3$ is finite, and $\Gamma_3/\Gamma_4$ is cyclic with $\cd_k(\Gamma_4) = n - 2$. Continuing in this way proves item \ref{item:subnormal}.

    \smallskip

    To prove \ref{item:subnormal} implies \ref{item:weak_dim}, it suffices to prove the following claim: If 
    \[
        1 \longrightarrow N \longrightarrow G \longrightarrow Q \longrightarrow 1
    \]
    is an extension with $G$ RFRS, $Q$ amenable, and $\Dk{N}$ of weak dimension at most $d$ as a $k[N]$-module, then $\Dk{G}$ is of weak dimension at most $d$ as a $k[G]$-module. The argument is similar to that of \cite[Theorem 3.4]{JaikinLinton_coherence}. Indeed, let $M$ be any $k[G]$-module. Then, for all $i > d$,
    \[
        \Tor_i^{k[G]}(\Dk{N} * G/N,M) \ \cong \ \Tor_i^{k[N]}(\Dk{N},M) \ = \ 0
    \]
    by Shapiro's lemma. But $\Dk{G}$ is the Ore localisation of $\Dk{N}*G/N$ by a result of Tamari \cite{Tamari_Ore}, and since Ore localisation is flat, we also get $\Tor_i^{k[G]}(\Dk{G},M) = 0$ for all $i > d$, as claimed. \qedhere
\end{proof}

\begin{cor}\label{cor:extension_of_free}
    Let $G$ be a RFRS group of type $\FP(k)$. Then $\Dk{G}$ is of weak dimension one as a $k[G]$-module if and only if $G$ is an iterated extension of a free group by (cyclic or finite) groups.
\end{cor}

\begin{rem}
    In the introduction, \cref{cor:extension_of_free} was stated in terms of $\mathcal U(G)$, the algebra of operators affiliated to the von Neumann algebra of $G$. Let $k \subseteq \C$ be a field. Since $\mathcal D_{k[G]}$ is the division closure of $k[G]$ in $\mathcal U(G)$, and moreover $\mathcal D_{k[G]}$ is a division ring, $\mathcal U(G)$ is faithfully flat over $\mathcal D_{k[G]}$. It follows that $\mathcal U(G)$ has weak dimension at most one as a $k[G]$-module if and only if $\mathcal D_{k[G]}$ does.
\end{rem}

It is an interesting problem to determine the difference between the invariants $b_i^{(2)}(G;k)$ for different fields $k$. In the context of RFRS groups (of finite type, say), it is an easy consequence of the results main results of \cite{KielakRFRS} and \cite{Fisher_Improved} that $b_1^{(2)}(G;k) = 0$ for one field $k$ if and only if $b_1^{(2)}(G;k') = 0$ for all fields $k'$. It is open whether the value of $b_1^{(2)}(G;k)$ depends on the field $k$, but for every $i > 1$ there are examples of right-angled Artin groups $G$ (which are, in particular, RFRS) such that $b_i^{(2)}(G;k)$ depends on $k$ \cite{AvramidiOkunSchreve2021}. Here we show that $\Dk{G}$ being of weak dimension one (or equivalently, that all subgroups of $G$ have vanishing $k$-$L^2$-Betti numbers) does not depend on $k$.

\begin{cor}
    Let $G$ be a RFRS group of type $\FP$. If $\Dk{G}$ is of weak dimension one as a $k[G]$-module for some field $k$, then $\mathcal D_{k'[G]}$ is of weak dimension one as a $k'[G]$-module for all fields $k'$.
\end{cor}

We conclude the section with applications to groups with elementary hierarchies. Following Hagen and Wise \cite{HagenWise_freebyZ}, we say that the trivial group has an elementary hierarchy of length $0$, and that a group $G$ has an elementary hierarchy of length $n$ if it splits as a graph of groups with an elementary hierarchy of length $n-1$ with cyclic edge groups. The main theorem of \cite{HagenWise_freebyZ} is that any finitely generated virtually special group with an elementary hierarchy is virtually free-by-cyclic (a group is \emph{special} if and only if it is a subgroup of a RAAG, so special groups are RFRS). This class of groups includes, e.g., graphs of free groups with cyclic edge groups not containing unbalanced Baumslag--Solitar groups and limit groups not containing $\Z^3$-subgroups. 

We extend the Hagen--Wise theorem in the following way: say that $G$ has an Abelian hierarchy of length $0$ if it is trivial and that it has an Abelian hierarchy of length $n$ if it virtually splits as a graph of groups with an Abelian hierarchy of length $n-1$ and finitely generated virtually Abelian edge groups. While all groups with an elementary hierarchy of cohomological dimension at most two, groups with an Abelian hierarchy can be of arbitrarily high cohomological dimension (the class contains all limits groups, for example).

\begin{cor}\label{cor:Abelian_hierarchy}
    Let $G$ be a virtually RFRS group of type $\FP(k)$ for some field $k$ admitting a virtually Abelian hierarchy. Then $G$ is an iterated extension of a free group by (cyclic or finite) groups.
\end{cor}
\begin{proof}
    It suffices to prove that every subgroup of $G$ has vanishing $k$-$L^2$-Betti numbers above degree one. Every subgroup of $G$ admits an Abelian hierarchy, so it suffices to prove the claim for a fixed group $\Gamma$ with an Abelian hierarchy of length $n$. If $n = 0$, then $\Gamma$ is trivial, and therefore $\b{i}(\Gamma;k) = 0$ for all $i > 0$. Now assume $n > 0$; since $k$-$L^2$-Betti numbers scale under passage to finite-index subgroups, we will assume that $\Gamma$ admits a graph of groups splitting $(\Gamma_v, \Gamma_e)$, where each $\Gamma_v$ has an Abelian hierarchy of length $n-1$ and $\Gamma_e$ is virtually Abelian. Then Chiswell's long exact sequence in the homology of a graph of groups \cite{ChiswellMV1976} and the fact that virtually Abelian groups have vanishing $k$-$L^2$-Betti numbers above degree zero imply that $\b{i}(\Gamma;k) = 0$ for all $i > 1$. The result now follows from \cref{cor:extension_of_free}. \qedhere
\end{proof}

\section{Novikov cohomology versus Novikov homology}\label{sec:one_rel}

Let $G$ be a group of finite type and of cohomological dimension $n$ and let $\chi \colon G \rightarrow \Z$ be an integral character. If for some ring $R$ we have $\H^n(G;\nov{R[G]}{\chi}) = 0$, then $\H_n(G;\nov{R[G]}{\chi}) = 0$. A natural question is whether the converse holds. if the converse fails, it still makes sense to ask whether the vanishing of top-dimensional Novikov homology controls some property of the kernel. A natural guess is that if $\H_n(G;\nov{R[G]}{\pm\chi}) = 0$, then the homological dimension of $\ker \chi$ over $R$ is $n-1$. In this brief section, we will see that these questions have negative answers.

\begin{lem}\label{lem:one_rel_vanish}
    Let $G$ be a torsion-free one-relator group and let $R$ be any ring without zero divisors containing $\Z[G]$. Then $\H_2(G;R) = 0$.
\end{lem}
\begin{proof}
    Since $G$ is torsion-free, $\Z$ admits a projective resolution of the form
    \[
        0 \longrightarrow \Z[G] \xlongrightarrow{\partial} \Z[G]^n \longrightarrow \Z[G] \longrightarrow \Z \longrightarrow 0
    \]
    (this is a consequence of the Lyndon identity theorem \cite{Lyndon_identityThm}).
    We claim that the extended map $\partial \colon R \rightarrow R^n$ is injective. Indeed, suppose $\partial(r) = r \partial(1) = 0$. Since $\partial$ is injective, $\partial(1)$ is nonzero. Since $R$ has no zero divisors, this implies $r = 0$, as desired. \qedhere
\end{proof}

\begin{cor}\label{cor:Novikov_one_rel_homol}
    If $G$ is a torsion-free one-relator group and $\chi \colon G \rightarrow \R$ is any homomorphism, then $\H_2(G;\nov{\Z[G]}{\chi}) = 0$.
\end{cor}
\begin{proof}
    One-relator groups are locally indicable by a result of Brodski\u{\i} \cite{BrodskiiOR} and therefore the group ring $\Z[G]$ has no zero divisors by a result of Higman \cite{Higman_thesis}. By looking at terms of minimal $\chi$-height, it is easy to show that $\nov{\Z[G]}{\chi}$ also has no zero divisors as well. The result then follows immediately from \cref{lem:one_rel_vanish}.
\end{proof}

\begin{ex}
    Consider the Baumslag--Solitar group 
    \[
        G \ = \ \BS(2,3) \ = \ \langle a,t \mid t\inv a^2 t \ = \ a^3 \rangle.
    \]
    It is well known that $G$ is not residually finite, and therefore it cannot be virtually free-by-cyclic \cite{Baumslag_freeByCyclicRF}. So for any finite-index subgroup $H \leqslant G$ and any epimorphism $\chi \colon H \rightarrow \Z$, either
    \[
        \H^2(H;\widehat{\Z[H]}^\chi) \ \neq \ 0 \quad \text{or} \quad \H^2(H;\widehat{\Z[H]}^{-\chi}) \ \neq \ 0
    \]
    by \cref{introcor:if_and_only_if}. In particular, let $\chi \colon G \rightarrow \Z$ be the map defined by $\chi(a) = 0$ and $\chi(t) = 1$. With \cref{cor:Novikov_one_rel_homol} this gives an example of a group of finite type and a map with non-vanishing top-dimensional Novikov cohomology that has vanishing Novikov homology. Moreover, $\ker \chi$ is not of homological dimension one. Indeed, there is a well defined homomorphism
    \[
        \pres{x,y}{x^3 = y^2} \longrightarrow G, \qquad x \longmapsto a, \quad y \longmapsto t\inv a t,
    \]
    which is injective (we leave checking injectivity as an exercise for the reader). Since $M \cong \pres{x,y}{x^3 = y^2}$ is a non-free one-relator group, it is of homological dimension two, which shows that $\ker \chi$ is of homological dimension two. Hence, vanishing top-dimensional Novikov homology does not imply that the homological dimension of the kernel drops.
\end{ex}

\bibliography{bib.bib}
\bibliographystyle{alpha}

\end{document}

%% file: top.tex
\usepackage{amsmath}
\usepackage{amsthm}
\usepackage{amssymb}
\usepackage{enumitem}\setlist[enumerate]{nosep,label=\textnormal{(\arabic*)}}
\usepackage[T1]{fontenc}
\usepackage[hidelinks]{hyperref}
\usepackage[capitalise]{cleveref}
\usepackage{mathrsfs}
\usepackage{mathtools}
\usepackage{nameref}
\usepackage[new]{old-arrows}
\usepackage{stmaryrd}
\usepackage{thmtools}
\usepackage[minimal]{yhmath}

\usepackage{tikz}
\usetikzlibrary{cd}
\usetikzlibrary{positioning, angles, quotes}

\theoremstyle{plain}
\newtheorem{thm}{Theorem}[section]
\newtheorem{cor}[thm]{Corollary}        
\newtheorem{prop}[thm]{Proposition}
\newtheorem{lem}[thm]{Lemma}
\newtheorem{claim}[thm]{Claim}
        \newtheorem*{exer*}{Exercise}
\newtheorem{conj}[thm]{Conjecture}      \newtheorem*{conj*}{Conjecture}
           \newtheorem*{q*}{Question}

\newtheorem{thmx}{Theorem}
\newtheorem{corx}[thmx]{Corollary}

\theoremstyle{definition}
\newtheorem{defn}[thm]{Definition}      \newtheorem*{defn*}{Definition}
\newtheorem{ex}[thm]{Example}           
\newtheorem{notation}[thm]{Notation}    \newtheorem*{notation*}{Notation}
\newtheorem{obs}[thm]{Observation}      \newtheorem*{obs*}{Observation}

\theoremstyle{remark}
\newtheorem{rem}[thm]{Remark}           \newtheorem*{rem*}{Remark}
\newtheorem{conv}[thm]{Convention}     \newtheorem*{conv*}{Convention}

\newtheoremstyle{iremark}
    {0.5em\topsep}   
    {0.5em\topsep}   
    {\upshape}  
    {0pt}       
    {\itshape}  
    {.}         
    {5pt plus 1pt minus 1pt} 
    {\thmname{#1}\thmnumber{ \itshape#2}\thmnote{ (#3)}} 
\theoremstyle{iremark}

\theoremstyle{plain}

\Crefname{thm}{Theorem}{Theorems}
\Crefname{lem}{Lemma}{Lemmas}
\Crefname{defn}{Definition}{Definitions}
\Crefname{claim}{Claim}{Claims}
\Crefname{ex}{Example}{Examples}
\Crefname{prop}{Proposition}{Propositions}

\newcommand{\BS}{\mathrm{BS}}

\newcommand{\pres}[2]{\langle{#1}\mid{#2}\rangle}

\newcommand{\FP}{\mathrm{FP}}

\renewcommand{\b}[1]{b^{(2)}_{#1}}

\newcommand{\Dk}[1]{\mathcal D_{k[{#1}]}}

\newcommand{\nov}[2]{\widehat{{#1}}^{#2}}

\newcommand{\C}{\mathbb{C}}

\newcommand{\Q}{\mathbb{Q}}
\newcommand{\R}{\mathbb{R}}
\newcommand{\Z}{\mathbb{Z}}

\newcommand{\inv}{^{-1}}

\DeclareMathOperator{\cd}{cd}
\DeclareMathOperator{\coInd}{coInd}

\let\H\relax
\DeclareMathOperator{\H}{H}
\DeclareMathOperator{\Hom}{Hom}
\DeclareMathOperator{\id}{id}
\DeclareMathOperator{\im}{im}
\DeclareMathOperator{\Ind}{Ind}

\DeclareMathOperator{\res}{res}

\DeclareMathOperator{\Tor}{Tor}

\makeatletter
\newsavebox{\@brx}
\newcommand{\llangle}[1][]{\savebox{\@brx}{\(\m@th{#1\langle}\)}%
  \mathopen{\copy\@brx\mkern2mu\kern-0.9\wd\@brx\usebox{\@brx}}}
\newcommand{\rrangle}[1][]{\savebox{\@brx}{\(\m@th{#1\rangle}\)}%
  \mathclose{\copy\@brx\mkern2mu\kern-0.9\wd\@brx\usebox{\@brx}}}
\makeatother

\newcounter{comments}

%% file: Novikov.bbl
\begin{thebibliography}{MMV98}

\bibitem[AOS21]{AvramidiOkunSchreve2021}
Grigori Avramidi, Boris Okun, and Kevin Schreve.
\newblock Mod {$p$} and torsion homology growth in nonpositive curvature.
\newblock {\em Invent. Math.}, 226(3):711--723, 2021.

\bibitem[Bau71]{Baumslag_freeByCyclicRF}
Gilbert Baumslag.
\newblock Finitely generated cyclic extensions of free groups are residually finite.
\newblock {\em Bull. Austral. Math. Soc.}, 5:87--94, 1971.

\bibitem[BG99]{BuxGonzalez1999}
Kai-Uwe Bux and Carlos Gonzalez.
\newblock The {B}estvina-{B}rady construction revisited: geometric computation of {$\Sigma$}-invariants for right-angled {A}rtin groups.
\newblock {\em J. London Math. Soc. (2)}, 60(3):793--801, 1999.

\bibitem[Bie07]{BieriDeficiency}
Robert Bieri.
\newblock Deficiency and the geometric invariants of a group.
\newblock {\em J. Pure Appl. Algebra}, 208(3):951--959, 2007.
\newblock With an appendix by Pascal Schweitzer.

\bibitem[BNS87]{BNSinv87}
Robert Bieri, Walter~D. Neumann, and Ralph Strebel.
\newblock A geometric invariant of discrete groups.
\newblock {\em Invent. Math.}, 90(3):451--477, 1987.

\bibitem[BR88]{BieriRenzValutations}
Robert Bieri and Burkhardt Renz.
\newblock Valuations on free resolutions and higher geometric invariants of groups.
\newblock {\em Comment. Math. Helv.}, 63(3):464--497, 1988.

\bibitem[Bro75]{Brown_HomCritFinite}
Kenneth~S. Brown.
\newblock Homological criteria for finiteness.
\newblock {\em Comment. Math. Helv.}, 50:129--135, 1975.

\bibitem[Bro84]{BrodskiiOR}
Sergei~D. Brodski\u{\i}.
\newblock Equations over groups, and groups with one defining relation.
\newblock {\em Sibirsk. Mat. Zh.}, 25(2):84--103, 1984.

\bibitem[Bro94]{BrownGroupCohomology}
Kenneth~S. Brown.
\newblock {\em Cohomology of groups}, volume~87 of {\em Graduate Texts in Mathematics}.
\newblock Springer-Verlag, New York, 1994.
\newblock Corrected reprint of the 1982 original.

\bibitem[Chi76]{ChiswellMV1976}
I.~M. Chiswell.
\newblock Exact sequences associated with a graph of groups.
\newblock {\em J. Pure Appl. Algebra}, 8(1):63--74, 1976.

\bibitem[Fel71]{Feldman71}
G.~L. Fel'dman.
\newblock The homological dimension of group algebras of solvable groups.
\newblock {\em Izv. Akad. Nauk SSSR Ser. Mat.}, 35:1225--1236, 1971.

\bibitem[FGS10]{FarberGeogheganSchutz_SikoravComplexes}
M.~Farber, R.~Geoghegan, and D.~Sch\"{u}tz.
\newblock Closed 1-forms in topology and geometric group theory.
\newblock {\em Uspekhi Mat. Nauk}, 65(1(391)):145--176, 2010.

\bibitem[FIK25]{FisherItalianoKielak_PDfibre}
Sam~P. Fisher, Giovanni Italiano, and Dawid Kielak.
\newblock Virtual fibring of {P}oincar{\'e}-duality groups, 2025.
\newblock {\tt arXiv:2506.14666}.

\bibitem[Fis24a]{Fisher_Improved}
Sam~P. Fisher.
\newblock Improved algebraic fibrings.
\newblock {\em Compos. Math.}, 160(9):2203--2227, 2024.

\bibitem[Fis24b]{Fisher_freebyZ}
Sam~P. Fisher.
\newblock On the cohomological dimension of kernels of maps to \(\mathbb{Z}\), 2024.
\newblock {\tt arXiv:2403.18758}, (to appear in \emph{Geom. Topol.}).

\bibitem[Hig40]{Higman_thesis}
Graham Higman.
\newblock {\em Units in group rings}.
\newblock {DPhil}, University of Oxford, 1940.

\bibitem[HK07]{HillmanKochloukova_PDnCovers}
J.~A. Hillman and D.~H. Kochloukova.
\newblock Finiteness conditions and {${\rm PD}_r$}-group covers of {${\rm PD}_n$}-complexes.
\newblock {\em Math. Z.}, 256(1):45--56, 2007.

\bibitem[HW10]{HagenWise_freebyZ}
Mark Hagen and Daniel~T. Wise.
\newblock Special groups with an elementary hierarchy are virtually free-by-{$\mathbb Z$}.
\newblock {\em Groups Geom. Dyn.}, 4(3):597--603, 2010.

\bibitem[JZ21]{JaikinZapirain2020THEUO}
Andrei Jaikin-Zapirain.
\newblock The universality of {H}ughes-free division rings.
\newblock {\em Selecta Math. (N.S.)}, 27(4):Paper No. 74, 33, 2021.

\bibitem[JZL25]{JaikinLinton_coherence}
Andrei Jaikin-Zapirain and Marco Linton.
\newblock On the coherence of one-relator groups and their group algebras.
\newblock {\em Ann. of Math. (2)}, 201(3):909--959, 2025.

\bibitem[Kie20]{KielakRFRS}
Dawid Kielak.
\newblock Residually finite rationally solvable groups and virtual fibring.
\newblock {\em J. Amer. Math. Soc.}, 33(2):451--486, 2020.

\bibitem[Lat94]{Latour_1formes}
Fran\c{c}ois Latour.
\newblock Existence de {$1$}-formes ferm\'ees non singuli\`eres dans une classe de cohomologie de de {R}ham.
\newblock {\em Inst. Hautes \'Etudes Sci. Publ. Math.}, (80):135--194 (1995), 1994.

\bibitem[Lin93]{LinnellDivRings93}
Peter~A. Linnell.
\newblock Division rings and group von {N}eumann algebras.
\newblock {\em Forum Math.}, 5(6):561--576, 1993.

\bibitem[L{\"u}c02]{Luck02}
Wolfgang L{\"u}ck.
\newblock {\em {$L\sp 2$}-invariants: theory and applications to geometry and {$K$}-theory}.
\newblock Springer-Verlag, Berlin, 2002.

\bibitem[Lyn50]{Lyndon_identityThm}
Roger~C. Lyndon.
\newblock Cohomology theory of groups with a single defining relation.
\newblock {\em Ann. of Math. (2)}, 52:650--665, 1950.

\bibitem[MMV98]{MeierMeinertVanWyk1998}
John Meier, Holger Meinert, and Leonard VanWyk.
\newblock Higher generation subgroup sets and the {$\Sigma$}-invariants of graph groups.
\newblock {\em Comment. Math. Helv.}, 73(1):22--44, 1998.

\bibitem[Ran95]{Ranicki_NovikovFiniteDomination}
Andrew Ranicki.
\newblock Finite domination and {N}ovikov rings.
\newblock {\em Topology}, 34(3):619--632, 1995.

\bibitem[Sch02]{SchickL2Int2002}
Thomas Schick.
\newblock Erratum: ``{I}ntegrality of {$L^2$}-{B}etti numbers''.
\newblock {\em Math. Ann.}, 322(2):421--422, 2002.

\bibitem[Sik87]{SikoravThese}
Jean-Claude Sikorav.
\newblock {\em Homologie de Novikov associ\'ee \`a une classe de cohomologie de degr\'e un}.
\newblock Th\`ese {d'\'Etat}, Universit\'e Paris-Sud (Orsay), 1987.

\bibitem[Sta68]{Stallings_cd1}
John~R. Stallings.
\newblock On torsion-free groups with infinitely many ends.
\newblock {\em Ann. of Math. (2)}, 88:312--334, 1968.

\bibitem[Swa69]{Swan_cd1}
Richard~G. Swan.
\newblock Groups of cohomological dimension one.
\newblock {\em J. Algebra}, 12:585--610, 1969.

\bibitem[Tam54]{Tamari_Ore}
Dov Tamari.
\newblock A refined classification of semi-groups leading to generalised polynomial rings with a generalised degree concept.
\newblock In {\em Proceedings of the {I}nternational {C}ongress of {M}athematicians, 1954, {A}msterdam}, volume III, pages 439--440. Erven P. Noordhoff N. V., Groningen; North-Holland Publishing Co., Amsterdam, 1954.

\bibitem[Wat60]{Watts_EWtheorem}
Charles~E. Watts.
\newblock Intrinsic characterizations of some additive functors.
\newblock {\em Proc. Amer. Math. Soc.}, 11:5--8, 1960.

\end{thebibliography}
